\documentclass[12pt]{amsart}
\usepackage[dvipsnames]{xcolor}
\usepackage{ulem}
\usepackage[colorlinks=true]{hyperref}
\usepackage{amsmath, amssymb,mathrsfs}
\usepackage{mathtools}
\usepackage{tkz-berge}
\usetikzlibrary{decorations.pathreplacing}
\usepackage{caption}
\usepackage{subcaption}
\usepackage[pagewise]{lineno}
\usepackage{tikz}
\usepackage{amssymb}
\usepackage{amsmath}
\usepackage{mathrsfs}
\usepackage{mathbbol}
\usepackage{amsfonts}
\usepackage{amssymb,amsmath}
\usepackage{graphicx}
\graphicspath{{./images/}}
\def\Z {{\mathbb Z}}
\def\R {\mathbb{R}}
\def\T {\mathbb{T}}
\def\N {\mathbb{N}}
\def\d{{\,\rm d}}

\newtheorem{proposition}{Proposition}[section]
\newtheorem{theorem}[proposition]{Theorem}

\newtheorem{lemma}[proposition]{Lemma}
\theoremstyle{definition}
\newtheorem{definition}[proposition]{Definition}

\newtheorem{remark}[proposition]{Remark}
\numberwithin{equation}{section}

\begin{document}

\title[Observability of dispersive equations]{Observability of dispersive equations from line segments on the torus}

\author[Y. Wang]{Yunlei Wang$^{1}$}

\author[M. Wang]
{Ming Wang$^2$}

\thanks{1.Institut de Math\'{e}matiques de Bordeaux, Universit\'{e} de Bordeaux 351, cours de la Lib\'{e}ration, F 33405 TALENCE cedex\\	
2.School of Mathematics and Statistics, HNP-LAMA, Central South University, Hunan, Changsha, 410083, PR China\\	
  Emails: yunlei.wang@math.u-bordeaux.fr; m.wang@csu.edu.cn(Corresponding author.)
}



\begin{abstract}
  We investigate the observability of a general class of linear dispersive equations on the torus $\mathbb{T}$. We take one line segment or two line segments in space-time region as the observable set. We give the characteristic on the slopes of the line segments to  guarantee the qualitative observability and quantitative observability respectively. The one line segment case, is simple, follows directly from the Ingham's inequality. However, the two line segments case is difficult, the statement of results and the proof  rely  heavily on the language of graph theory. We also apply our results to (higher order) Schr\"{o}dinger equations and the linear KdV equation.
\end{abstract}

\maketitle

\section{Introduction}
\label{sec:1}

\subsection{Background and the problem}
Let $\T:=\R / 2\pi\Z$ be the one dimensional torus. Consider the Schr\"{o}dinger equation on $\T$
\begin{equation}\label{first}
     \partial_t u=i\partial_x^2u, \quad  u(0,x)\in L^2(\T).
\end{equation}
Then for every nonempty open set $G\subset \R\times \T$, the observability inequality
\begin{equation}\label{equ-ob-open-space-time}
      \int_{\T}|u(0,x)|^2\d x\le C(G)\int_G |u(t,x)|^2\d x\mathrm{d}t,
\end{equation}
holds for all solutions of \eqref{first}, where $C(G)>0$ is a constant depending only on $G$. This is a classical result, see e.g. \cite{haraux1989series,jaffard1990controle,komornik1992exact}, \cite[pp.~166-172]{komornik2005fourier},  and \cite[pp.~280-284]{tucsnak2009observation} for a proof.
Roughly speaking, \eqref{equ-ob-open-space-time} means that one can recover the solution by observation on the region $G$.
In particular, \eqref{equ-ob-open-space-time} implies that
\begin{equation}\label{equ-ob-open-space-1}
    \int_{\T}|u(0,x)|^2\d x\le C(T,\omega)\int_0^T\int_\omega |u(t,x)|^2\d x\mathrm{d}t,
\end{equation}
where $T>0$, $\omega \subset \T$ is a nonempty open set. Burq and Zworski \cite{burq2019rough} further showed that \eqref{equ-ob-open-space-1} still holds if $\omega \subset \T$ is a Lebesgue measurable set with positive measure.
We refer to \cite{anantharaman2016wigner,bourgain2013control,huang2022observable,rosier2009exact,tao2021exact}  for more results on the observability inequality for Schr\"{o}dinger equations with potentials and in higher dimensions.

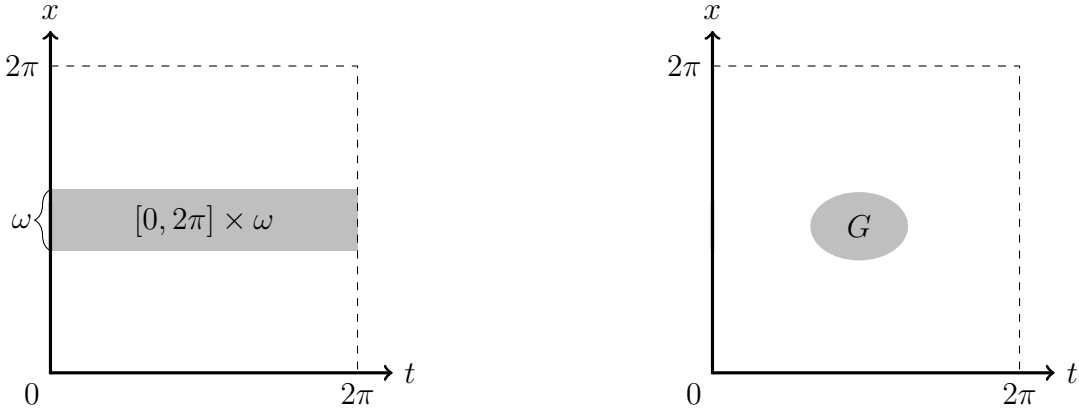
\begin{figure}[htbp]
    \centering
    \begin{subfigure}[b]{0.45\textwidth}
        \begin{tikzpicture}[scale=0.65]
           \fill[gray!50] (0,2.5)--(0,3.75)--(2*3.14,3.75)--(2*3.14,2.5) -- cycle;
           \draw[very thick,->] (-0,0) -- (7,0) node[right] {$t$};
           \draw[very thick,->] (0,-0.03) -- (0,7) node[above]{$x$};
           \draw (0,0) node[below left]{$0$};
           \draw (0,3.14*2) node[left]{$2\pi$};
           \draw (3.14*2,0) node[below]{$2\pi$};
           \draw[dashed] (0,3.14*2) -- (3.14*2,3.14*2) -- (3.14*2,0);
           \draw[very thick] (0,2.5) -- (0,3.5);
           \draw[decorate, decoration={brace,amplitude=6pt}] (0,2.5)--(0,3.75);
           \draw (-0.15,3.1) node[left] {$\omega$};
           \draw (3.14,3.1) node {$[0,2\pi]\times \omega$};
    \end{tikzpicture}
    \end{subfigure}
    \hfill
    \begin{subfigure}[b]{0.45\textwidth}
            \begin{tikzpicture}[scale=0.65]
           \draw[very thick,->] (-0,0) -- (7,0) node[right] {$t$};
           \draw[very thick,->] (0,-0.03) -- (0,7) node[above]{$x$};
           \draw (0,0) node[below left]{$0$};
           \draw (0,3.14*2) node[left]{$2\pi$};
           \draw (3.14*2,0) node[below]{$2\pi$};
           \draw[dashed] (0,3.14*2) -- (3.14*2,3.14*2) -- (3.14*2,0);
           \draw[very thick] (0,2.5) -- (0,3.5);
           \fill[gray!50] (3,3) ellipse (1 and 0.7);
           \draw (3,3) node {$G$};
    \end{tikzpicture}
    \end{subfigure}
    \caption{Observability from a positive measure set in $[0,2\pi]^2$.}
    \label{fig:1}
\end{figure}
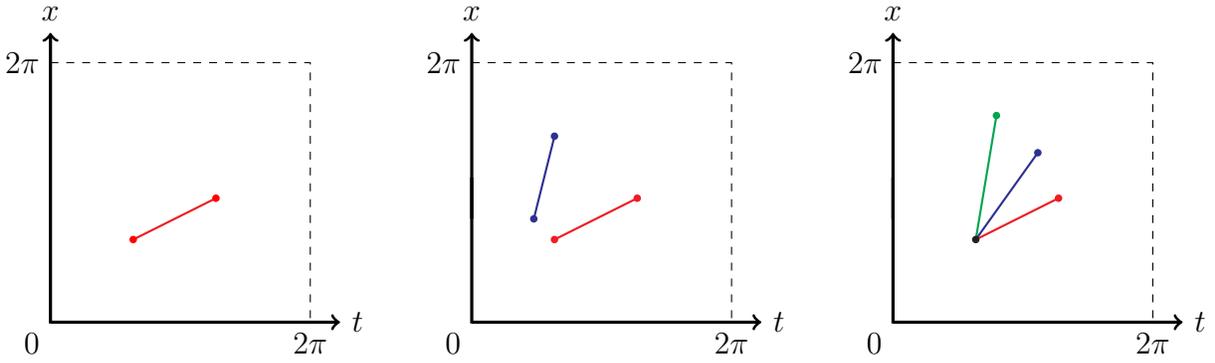
\begin{figure}[htbp]
    \centering
    \begin{subfigure}[b]{0.3\textwidth}
        \begin{tikzpicture}[scale=0.55]
           \draw[very thick,->] (-0,0) -- (7,0) node[right] {$t$};
           \draw[very thick,->] (0,-0.03) -- (0,7) node[above]{$x$};
           \draw (0,0) node[below left]{$0$};
           \draw (0,3.14*2) node[left]{$2\pi$};
           \draw (3.14*2,0) node[below]{$2\pi$};
           \draw[dashed] (0,3.14*2) -- (3.14*2,3.14*2) -- (3.14*2,0);
           \draw[Red,thick,opacity=0.6] (2,2)--(4,3);
           \draw (2,2)node[circle, fill,red, inner sep=1pt]{};
           \draw (4,3)node[circle, fill,red, inner sep=1pt]{};
    \end{tikzpicture}
    \end{subfigure}
    \hfill
    \begin{subfigure}[b]{0.3\textwidth}
            \begin{tikzpicture}[scale=0.55]
           \draw[very thick,->] (-0,0) -- (7,0) node[right] {$t$};
           \draw[very thick,->] (0,-0.03) -- (0,7) node[above]{$x$};
           \draw (0,0) node[below left]{$0$};
           \draw (0,3.14*2) node[left]{$2\pi$};
           \draw (3.14*2,0) node[below]{$2\pi$};
           \draw[dashed] (0,3.14*2) -- (3.14*2,3.14*2) -- (3.14*2,0);
           \draw[very thick] (0,2.5) -- (0,3.5);
           \draw[Red,thick,opacity=0.6] (2,2)--(4,3);
           \draw[Blue,opacity=0.6,thick] (1.5,2.5)--(2,4.5);
           \draw (2,2)node[circle, fill,Red, inner sep=1pt]{};
           \draw (4,3)node[circle, fill,Red, inner sep=1pt]{};
           \draw (1.5,2.5)node[circle, fill,Blue, inner sep=1pt]{};
           \draw (2,4.5)node[circle, fill,Blue, inner sep=1pt]{};
    \end{tikzpicture}
    \end{subfigure}
    \hfill
    \begin{subfigure}[b]{0.3\textwidth}
            \begin{tikzpicture}[scale=0.55]
           \draw[very thick,->] (-0,0) -- (7,0) node[right] {$t$};
           \draw[very thick,->] (0,-0.03) -- (0,7) node[above]{$x$};
           \draw (0,0) node[below left]{$0$};
           \draw (0,3.14*2) node[left]{$2\pi$};
           \draw (3.14*2,0) node[below]{$2\pi$};
           \draw[dashed] (0,3.14*2) -- (3.14*2,3.14*2) -- (3.14*2,0);
           \draw[very thick] (0,2.5) -- (0,3.5);
           \draw[Red,thick,opacity=0.6] (2,2)--(4,3);
           \draw[Blue,thick,opacity=0.6] (2,2)--(3.5,4.1);
           \draw[Green,thick,opacity=0.6] (2,2)--(2.5,5);
           \draw (2,2)node[circle, fill,Black, inner sep=1pt]{};
           \draw (4,3)node[circle, fill,Red, inner sep=1pt]{};
           \draw (2.5,5)node[circle, fill,Green, inner sep=1pt]{};
           \draw (3.5,4.1)node[circle, fill,Blue, inner sep=1pt]{};
    \end{tikzpicture}
    \end{subfigure}
    \caption{Observability from one line segments holds if the slope is not an integer (see left). Observability from two line segments fails if the slopes are integers (see middle). Observability from $N\ge 3$ line segments fails if the slopes are integers (see right and note that they have the common starting point).}
    \label{fig:2}
\end{figure}

Note that the observation region mentioned above is at least a set with positive measure in $\R\times \T$, see Figure \ref{fig:1}. Recently, Jaming and Komornik \cite{jaming2020moving} introduced a new observability inequality for \eqref{first}, with observation from line segments in the space-time plane. The observation regions are essentially smaller than that in previous references since the two dimensional Lebesgue measure is zero now.  In fact, it is shown in  \cite{jaming2020moving} that, for every $t_0,x_0\in \R$,  the observability inequality from $N$ segments (with the same starting point if $N\ge 3$)
\begin{equation}\label{equ-ob-line-sch}
    \int_\T |u(0,x)|^2\d x\le C(T,v)\sum_{i=1}^{N}\int_0^T|u(t_0+t,x_0-v_i t)|^2\d t,
\end{equation}
holds for all solutions to \eqref{first} and all $T>0$ if and only if there exists $v_i \in \R \backslash \Z$ for some $i \in \left\{1,2,\cdots,N\right\}$, see Figure~\ref{fig:2}.
The proof follows from various kinds of Ingham's theorems and the algebra structure of the symbol $k^2$ associated to the Schr\"{o}dinger equation. A closely related work of Bourgain and Rudnick \cite{bourgain2012Restriction} shows that, the $L^2$ norm of eigenfunctions to the Laplacian on the torus $\T^d(d\leq 3)$ can be controlled by its restrictions on surfaces with non-vanishing curvature.

The inequality \eqref{equ-ob-line-sch} is called observability inequality with moving points in \cite{jaming2020moving}.
We refer to \cite{capistrano2020pointwise,jaming2018dynamical,khapalov1995exact,khapalov2017mobile,micu2022fractional} for more observability from moving points, and to \cite{beauchard2020geometric,chaves2014null,le2017geometric,martin2013null}  for observability from moving observation domain in higher dimensions.

Inspired by \cite{jaming2020moving}, we consider a family of linear dispersive equations on $\T$, and investigate its qualitative and quantitative observability  from line segments.
More precisely, we consider the dispersive equation on $\T$
\begin{equation}\label{eqn-abstract}
    \partial_t u= i P(D)u,\quad  u(0,x)=u_0(x)\in L^2(\T),
\end{equation}
where $u=u(t,x)$ denotes a real or complex-valued function, $P(D)$ denotes a Fourier multiplier operator, $D=i^{-1}\partial_x$. Let $p:\Z\to \R$ be the symbol of $P(D)$, namely
\begin{equation}
    \widehat{P(D)u}(k):=p(k)\widehat{u}(k),\quad k\in \Z,\label{eqn-condition-1}
\end{equation}
where $\widehat{u}(k)$ stands for the $k$-th Fourier coefficient of $u$
\begin{equation}
    \widehat{u}(k)=\frac{1}{2\pi }\int_0^{2\pi} u(x)e^{-ikx}\d x.
\end{equation}
{It should be noted that $p$ need not be a polynomial.} Since $p$ is real-valued, $iP(D)$ is anti-symmetric on $L^2(\T)$. By Stone's theorem \cite{pazy2012semigroups}, $iP(D)$ generates a unitary group $\{e^{itP(D)}\}_{t\in \R}$ on $L^2(\R)$, thus \eqref{eqn-abstract} is well-posed in $L^2(\T)$ and
$\|u(t,\cdot)\|_{L^2(\T)}=\|u(0,\cdot)\|_{L^2(\T)}, \forall t\in \R.$
The controllability and stabilization of \eqref{eqn-abstract} has been studied in \cite{pastor2022two}.

Let $T>0, n\in \N:=\lbrace 1,2,3,\cdots\rbrace$ and $(t_i,x_i)\in \R^2,i=1,2,\cdots,n$ be given. We are mainly interested in the following two problems:
\vspace{2.6mm}

\noindent {\bf Qualitative observability problem:} Find out the characterization of $v_1,v_2,\cdots,$ $v_n\in \R$ such that the unique continuation property
\begin{equation}\label{eqn-qualitative}
    u(t_i+t,x_i-v_it)=0,\forall i=1,2,\cdots,n,t\in (0,T)\Longrightarrow u_0=0
\end{equation}
holds for all solutions $u(t,x)$ of \eqref{eqn-abstract}.

\noindent {\bf Quantitative observability problem:} Find out the characterization of $v_1,v_2,\cdots,$ $v_n\in \R$ such that  the observability inequality
\begin{equation}\label{eqn-quantitative}
    \int_\T|u_0(x)|^2\d x\lesssim \int_0^T\sum_{i=1}^n |u(t_i+t,x_i- v_i t)|^2\d  t
\end{equation}
holds for all solutions $u(t,x)$ of \eqref{eqn-abstract}.

For some technical reasons, we shall focus on the cases $n=1,2$ in this paper (see Remark~\ref{rmk2} (iv) in detail).

\subsection{Some notations}
\label{subsec:12}
Throughout, we use $A\lesssim B$ to denote $A\le CB$ for some unharmful constant $C>0$. We write $A\asymp B$ if both $A\lesssim B$ and $B\lesssim A$ hold.

For every $u_0\in L^2(\T)$, we have the Fourier expansion
\begin{equation}\label{equ-data}
    u_0(x)=\sum_{k\in\Z} c_ke^{ikx},
\end{equation}
where $c_k$ is the Fourier coefficient of $u_0$. Then, by Fourier transform, the solution $u$ of \eqref{eqn-abstract} is given by
\begin{equation}\label{equ-abstract-solu}
    u(t,x)=\sum_{k\in\Z}c_ke^{ip(k) t}e^{ikx}.
\end{equation}
Given $v\in\R$. By \eqref{equ-abstract-solu}, $u(t_0+t,x_0-vt)$ can be expressed as
\begin{equation}\label{eqn-one-point-representation}
    u(t_0+t,x_0-vt)=\sum_{k\in\Z}c_ke^{ip(k)(t_0+t)}e^{ik(x_0-vt)}=\sum_{k\in\Z} c_ke^{i(p(k)t_0+kx_0)} e^{i\lambda_k(v)t},
\end{equation}
where $\lambda_k(v):=p(k)-kv$. For each $k\in \Z$, we set
\begin{equation}
    \Xi_k(v):=\bigl\lbrace m\in\Z : \lambda_m(v)=\lambda_k(v)\bigr\rbrace
\end{equation}
and $n_k(v):=|\Xi_k(v)|$ denotes the cardinality. {If $p$ is a polynomial, one has $n_k(v)\leq \mbox{deg } p$ and even $n_k(v)\leq 2$ if $|k|$ is large enough, see Figure \ref{fig:add}.}

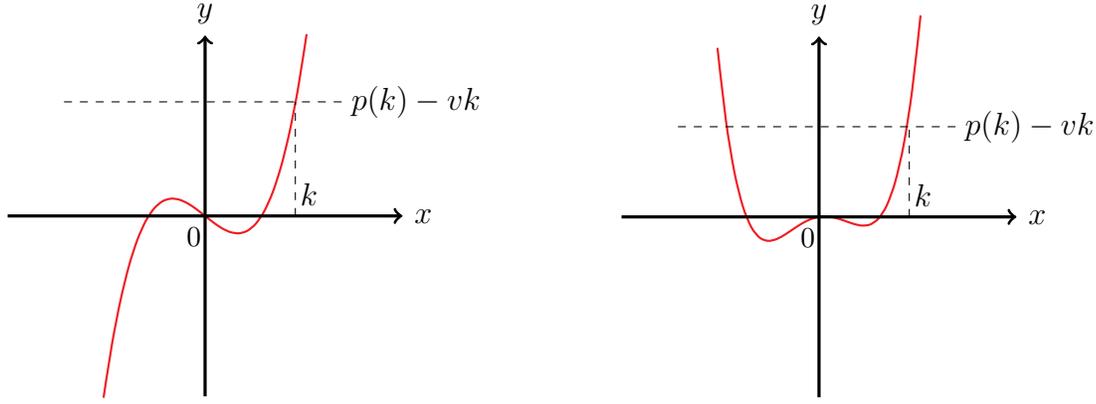
\begin{figure}[htbp]
        \centering
            \begin{subfigure}[b]{0.49\textwidth}
         \begin{tikzpicture}[xscale=0.75,yscale=0.6]
         \draw[Red,thick,opacity=0.6] plot[domain=-1.8:1.8, smooth] (\x, \x*\x*\x-1*\x);
           \draw[very thick,->] (-3.5,0) -- (3.5,0) node[right] {$x$};
           \draw[very thick,->] (0,-4) -- (0,4) node[above] {$y$};
           \draw (-0.2,-0.01) node[below] {\small $0$};
         \draw[dashed] (1.6,0) -- (1.6,2.53);
         \draw[dashed] (-2.5,2.53) -- (2.5,2.53);
         \draw (1.5,0) node[above right] {$k$};
         \draw (2.4,2.53) node[right] {$p(k)-vk$};
        \end{tikzpicture}
        \end{subfigure}
        \hfill
        \begin{subfigure}[b]{0.49\textwidth}
         \begin{tikzpicture}[xscale=0.75,yscale=0.6]
         \draw[Red,thick,opacity=0.6] plot[domain=-1.8:1.8, smooth] (\x, 0.7*\x*\x*\x*\x-1*\x*\x+0.2*\x);
           \draw[very thick,->] (-3.5,0) -- (3.5,0) node[right] {$x$};
           \draw[very thick,->] (0,-4) -- (0,4) node[above] {$y$};
           \draw (-0.2,-0.01) node[below] {\small $0$};
         \draw[dashed] (1.6,0) -- (1.6,2);
         \draw[dashed] (-2.5,2) -- (2.5,2);
         \draw (1.5,0) node[above right] {$k$};
         \draw (2.4,2) node[right] {$p(k)-vk$};
        \end{tikzpicture}            
        \end{subfigure}
        \caption{Graphs of polynomial function $y=p(x)-vx$: (1) if $\mathrm{deg}\,p$ is odd, then for $|k|$ large enough, $n_k(v)=1$ (see left); (2) if $\mathrm{deg}\,p$ is even, then for $|k|$ large enough, $n_k(v)\le 2$ (see right).}
        \label{fig:add}
    \end{figure}

We denote
\begin{equation}
    \Pi(v):=\bigl\lbrace k\in\Z : \exists m\in\Z,m\neq k \text{ such that } \lambda_k(v)=\lambda_m(v) \bigr\rbrace.\label{eqn-piv}
\end{equation}

Define
\begin{equation}\label{equ-Lambda(v)}
    \Lambda(v):=\lbrace \lambda_k(v) :k\in\Z \rbrace
\end{equation}
and
\begin{equation}\label{equ-gamma(v)-1}
        \gamma'(v):=\sup\limits_{W\subset \Lambda(v)}\inf\limits_{\substack{\lambda,\lambda'\in \Lambda(v)\backslash W\\ \lambda\neq \lambda'}}|\lambda-\lambda'|
\end{equation}
    where $W$ runs over all finite subsets of $\Lambda(v)$. In particular, if $|\Pi(v)|<\infty$, then
    \begin{equation}\label{equ-gamma(v)-2}
        \gamma'(v)=\sup\limits_{W\subset \Z }\inf\limits_{\substack{k,n\in \Z\backslash W\\ k\neq n}}|\lambda_k-\lambda_n|>0
    \end{equation}
    where $W$ runs over all finite subsets of $\Z$.

To present our results on the observability with observations from two line segments, we need to introduce some definitions further. Let $v_1,v_2\in \R$ be two different real numbers. We construct the graph $G(v_1,v_2)$ through the following rules:
\begin{itemize}
    \item the vertices of $G(v_1,v_2)$ equal to the set $\Pi(v_1)\cup \Pi(v_2)$;
    \item two distinct vertices $k_1$ and $k_2$ are adjacent through red (blue, respectively) color edge if $\lambda_{k_1}(v_1)=\lambda_{k_2}(v_1)$ ($\lambda_{k_1}(v_2)=\lambda_{k_2}(v_2)$, respectively).
\end{itemize}

 The graph $G(v_1,v_2)$ constructed here is first introduced by Jaming and Komornik \cite{jaming2020moving} (different but equivalent way), which is used to give some criteria for the observability of beam equations from two line segments.

 We claim that the graph $G(v_1,v_2)$ is simple. In other words, two distinct vertices $k_1$ and $k_2$ cannot have edges of both colors simultaneously. Indeed, if
$$
    \lambda_{k_1}(v_i)=\lambda_{k_2}(v_i)
$$
holds for $i=1,2$, then we obtain
$$
    (k_1-k_2)(v_1-v_2)=0,
$$
which is impossible for $k_1\neq k_2$ {since $v_1\neq v_2$} .

A \textit{path} is a simple graph whose vertices can be arranged in a linear sequence such that two vertices are adjacent if they are consecutive in the sequence and are nonadjacent otherwise. A \textit{cycle} is a simple graph whose vertices can be arranged in a cyclic sequence in such a way that two vertices are adjacent if they are consecutive in the sequence and are nonadjacent otherwise. The \textit{length} of a path or a cycle is the number of its edges.

There are two kinds of infinite path. A path who has an initial but no terminal vertex is called a \textit{one-way infinite path}. A path who has neither initial nor terminal vertex is called a \textit{two-way infinite path}. The length of both is defined to be $\infty$.

The above notions can be found in graph theory (see \cite{bondy2008graph} for more details).

\begin{definition}
A graph is called connected if there is a path from any point to any other point in the graph.
A connected component of a graph is a  connected subgraph that is not part of any larger connected subgraph.
\end{definition}
For a graph $G(v_1,v_2)$, we define
\begin{equation}
    g(v_1,v_2):=\max_{\substack{H\subset G(v_1,v_2)\\ \text{ a connected component}}}|H|,
\end{equation}
where $|H|$ denotes the number of vertices in $H$. The quantity $g(v_1,v_2)$ is important to illustrate that,  in what circumstance the qualitative observability property fails to be a quantitative one.

\begin{definition}
 A path (both finite and infinite) is said to be $2$-colored if its edges contains both colors.
\end{definition}

We shall use the following assumptions on the symbol $p(k)$.
\begin{enumerate}

    \item[(H1)] The set $\Xi_k(v)$ has only finitely many elements for all $k\in \Z$ and $v\in \R$, i.e.,
\begin{equation}
    n_k(v)<\infty,\quad \forall k\in \Z, v\in\R.\label{eqn-original-assump}
\end{equation}
    \item[(H2)] For each $v\in \R$, there exists $N_v\in\N$ such that
    \begin{equation}
        n_k(v)\le 2,\quad \forall k\ge N_v.
    \end{equation}
\end{enumerate}

\begin{remark}\label{rem-intro-1}
(i) Hypothesis (H1) ensures that each vertex in $G(v_1,v_2)$ has only finitely many neighborhoods.

(ii) Hypothesis (H2) ensures that all vertices but finite have at most one edge for each color. This prevents one-way infinite paths and two-way infinite paths in $G(v_1,v_2)$ to being too complicated.

(iii) The dispersive equations satisfying (H1)-(H2) are rather general.
In fact, { if $p$ is a polynomial or more generally} if $p\in C^1(\R)$ and $|p'(k)|\asymp |k|^\alpha$ as $k\to \infty$ for some $\alpha>0$, then (H1) and (H2) hold.
\end{remark}

\subsection{Main results}

Our main results, Theorem~\ref{thm:main2} and Theorem~\ref{thm:main3}, concern observability of dispersive equations from $n=2$ line segments. Before presenting these two results, we first consider the case $n=1$, namely the observability with observation from one line segment.
The {following} result describes the slope $v$ to ensure the observability inequality.
\begin{proposition}\label{thm:main1}
 Assume that $v\in \R,\gamma'(v)>0$. Then for every $(t_0,x_0)\in \R\times \T$, the observability inequality
\begin{align}\label{equ-ob-one-line}
\int_\T|u_0(x)|^2\d x\lesssim  \int_0^T|u(t_0+t,x_0-vt)|^2\d t
\end{align}
holds for all $T>\frac{2\pi}{\gamma'(v)}$ and all solutions to \eqref{eqn-abstract} if and only if $\Pi(v) = \O$.
\end{proposition}
The proof of Proposition~\ref{thm:main1} is easy, which follows from the Ingham's inequality directly. According to Proposition~\ref{thm:main1}, the observability \eqref{equ-ob-one-line} fails if $\Pi(v) \neq \O$. Also, the observability inequality from two line segments
\begin{align}\label{equ-ob-two-line}
\int_\T|u_0(x)|^2\d x\lesssim  \int_0^T\sum_{i=1}^2|u(t_0+t,x_0-v_it)|^2\d t
\end{align}
holds if one of $\Pi(v_1), \Pi(v_2)$ is empty. Thus it is natural to ask, whether \eqref{equ-ob-two-line} holds if $\Pi(v_1)\neq \O$ and $\Pi(v_2)\neq \O$? We answer the question in Theorem~\ref{thm:main2}. Indeed, we give the necessary and sufficient conditions on $v_1,v_2$ for both qualitative and quantitative observability of \eqref{eqn-abstract} as follows.
\begin{theorem} \label{thm:main2}
Let $(t_1,x_1),(t_2,x_2)\in \R\times \T$,  $v_1,v_2\in\R$, $v_1\neq v_2$, $\gamma'(v_1)>0,\gamma'(v_2)>0$, and $T>{2\pi} /{\min\bigl\lbrace\gamma'(v_1),\gamma'(v_2)\bigr\rbrace}$. Assume that {\rm{(H1)}} and {\rm{(H2)}} hold.
\begin{enumerate}
    \item [(i)] The qualitative observability 
{ $$
     u(t_i+t,x_i-v_it)=0,\forall i=1,2,t\in (0,T)\Longrightarrow u_0=0
 $$
 holds  for all solutions $u(t,x)$ of \eqref{eqn-abstract} } if and only if the graph $G(v_1,v_2)$ has no two-colored cycle.
    \item [(ii)] The quantitative observability 
    { 
    $$
    	\int_\T|u_0(x)|^2\d x\lesssim \int_0^T\sum_{i=1}^2 |u(t_i+t,x_i- v_i t)|^2\d  t
    $$
   holds for all solutions $u(t,x)$ of \eqref{eqn-abstract}
}    
    if and only if the graph $G(v_1,v_2)$ has no two-colored cycle and $g(v_1,v_2)<\infty$.
\end{enumerate}
\end{theorem}

\begin{remark}\label{rmk2} (i) The restriction $\gamma'(v)>0$ in Theorem~\ref{thm:main2} holds for  { many} kinds of dispersive equations. In particular, if $p(k)$ is a polynomial with degree $\mathrm{deg} \,p\geq 2$, then  we have $\gamma'(v)=\infty$ for every $v\in \R$, see Section~\ref{sec:4}.

(ii) By setting $p(k)=k^2$ in \eqref{eqn-abstract}, it is easy to see that the graph $G(v_1,v_2)$  has no two-colored cycle for different real numbers $v_1\neq v_2$, and $G(v_1,v_2)$ has a two-way infinite path for different integers $v_1\neq v_2\in \Z$. Thus, Theorem~\ref{thm:main2} recovers the Jaming-Komornik's results \cite{jaming2020moving}:
The qualitative observability \eqref{eqn-qualitative} of the Schr\"{o}dinger equation \eqref{first} holds for all $v_1,v_2\in\R$, $v_1\neq v_2$; but the quantitative observability \eqref{eqn-quantitative} fails for all $v_1,v_2\in\Z, v_1\neq v_2$.

(iii) The proof of Theorem~\ref{thm:main2} also relies on the Ingham's inequality. The central task here is to exploit the interaction of the two line segments' contribution in the observability inequality. To this end, the language of graph theory is widely used.
The novelty of Theorem~\ref{thm:main2} lies in two folds. First, it applies to a rather large number of dispersive equations, see {Remark~\ref{rem-intro-1} (iii)}. Second, it {clarifies} the essential difference between  the qualitative observability and the quantitative observability.

(iv) For the observability from $n\ge 3$ segments, based on our approach, we need to construct the graph with $n\ge 3$ kinds of colored edges. It is much more difficult to characterize the two properties in the language of graph theory.
\end{remark}

Now we apply Theorem~\ref{thm:main2} to more concrete equations.
\begin{theorem}\label{thm:main3}
Let $p$ be a real polynomial with degree $\mathrm{deg} \, p>2$. Let $T>0$, $(t_1,x_1),(t_2,x_2)\in \R\times \T$ and $v_1,v_2\in\R,v_1\neq v_2$.  Then the following are equivalent:
\begin{enumerate}
    \item[\rm{(i)}]   The following qualitative observability holds
    { $$
    	u(t_i+t,x_i-v_it)=0,\forall i=1,2,t\in (0,T)\Longrightarrow u_0=0.
    	$$   }
    \item[\rm{(ii)}]   The following quantitative observability  holds
    { 
    	$$
    	\int_\T|u_0(x)|^2\d x\lesssim \int_0^T\sum_{i=1}^2 |u(t_i+t,x_i- v_i t)|^2\d  t.
    	$$  }
    \item[\rm{(iii)}]  The graph $G(v_1,v_2)$ has no two-colored cycle.
\end{enumerate}
\end{theorem}

In Theorem~\ref{thm:main3} we show that, if $\mathrm{deg} \,p>2$, the qualitative observability and the quantitative observability are equivalent. However, the equivalence fails (see Jaming-Komornik's results \cite{jaming2020moving} mentioned above) {for} the linear Schr\"{o}dinger equation \eqref{first}, which satisfies \eqref{eqn-abstract} with $\mathrm{deg } \,p=2$. This gives a difference between the case $\mathrm{deg } \,p>2$ and $\mathrm{deg } \,p=2$.

Moreover, applying Theorem~\ref{thm:main3} to higher-order Schr\"{o}dinger equations and KdV equations, we obtain two typical examples.
\begin{itemize}
    \item Consider the higher-order linear Schr\"{o}dinger equation of the form
    \begin{equation*}
         \partial_tu=i(-\partial_x^2u+\partial_x^4 u+\cdots+(-1)^{l}\partial_x^{2l}u).
    \end{equation*}
    We prove that the quantitative observability always holds for all distinct $v_1,v_2\in \R$ (see Theorem~\ref{thm-application-1}). In other words, the observability inequality holds with observation from any two  non-parallel line segments.
    \item Consider the linear KdV equation of the form
    \begin{equation*}
        \partial_t u+\partial_x^3u=0.
    \end{equation*}
    Different from previous case, two non-parallel line segments is not sufficient for the observability anymore. However, we give two criteria on $v_1,v_2$ for the observability inequality: (1) $v_2>4v_1$; (2) some restrictions in terms of number theory, see Theorem~\ref{thm-application-2} for details.
\end{itemize}

\subsection{Outline of the work}

The paper is organized as follows. In Section~\ref{sec:2} we recall Ingham's inequality and prove Proposition~\ref{thm:main1}. We prove Theorem~\ref{thm:main2} and Theorem~\ref{thm:main3} in Section~\ref{sec:3} and Section \ref{sec:4}, respectively. Finally, we apply our results to Schr\"{o}dinger type equations and the KdV equation in Section~\ref{sec:5}, where we obtain Theorem~\ref{thm-application-1} and Theorem~\ref{thm-application-2}.

\section{Ingham inequality}\label{sec:2}
In this section, we first introduce some Ingham's theorems on non-harmonic Fourier series. Then we apply them to prove Proposition~\ref{thm:main1}.

\begin{definition}
A set of real numbers $\Lambda=\bigl\lbrace\lambda_k\bigr\rbrace_{k\in J}$  is called uniformly separated if
\begin{equation*}
    \gamma(\Lambda):=\inf_{\substack{k,n\in J\\k\neq n}}|\lambda_k-\lambda_n|>0.
\end{equation*}
The quantity $\gamma(\Lambda)$ is called the uniform gap of $\Lambda$.
\end{definition}
Ingham \cite{ingham1936some} proved the following important theorem.
\begin{theorem}[{\bf Ingham}]\label{thm-ingham-1}
Let $\Lambda=\lbrace \lambda_k \rbrace_{k\in J}\subset \R$ be uniformly separated.
\begin{enumerate}
    \item[\rm{(i)}] For every bounded interval $I$, the direct inequality
    \begin{equation*}
        \int_I|f(x)|^2\d x\lesssim \sum_{k\in J}|c_k|^2
    \end{equation*}
    holds for all functions of the form $f(x)= \sum_{k\in J}c_ke^{i\lambda_k x}$ with square-summable coefficients $c_k$.
    \item[\rm{(ii)}] If the interval $I$ has length $|I|>\frac{2\pi}{\gamma(\Lambda)}$, then the inverse inequality
    \begin{equation}
        \sum_{k\in J}|c_k|^2\lesssim \int_I| f(x) |^2 \d x\label{eqn-inverse-inequality}
    \end{equation}
    holds for all functions of the form $f(x)= \sum\limits_{k\in J}c_ke^{i\lambda_k x}$ with square-summable coefficients $c_k$.
\end{enumerate}
\end{theorem}

The condition $|I|>\frac{2\pi}{\gamma(\Lambda)}$ can be weakened in some circumstance. {Kahane \cite{Kahane1962Pse}  and} Haraux \cite{haraux1989series} proved the following generalized Ingham's Theorem.
\begin{theorem}[{\bf  Kahane-Haraux}]\label{thm-ingham-2}
Let $\Lambda=\bigl\lbrace \lambda_k\bigr\rbrace_{k\in J}$ be a uniformly separated family of real numbers. Then the inverse inequality \eqref{eqn-inverse-inequality} holds under the condition $|I|>\frac{2\pi}{\gamma'(\Lambda)}$, where
\begin{equation}
    \gamma'(\Lambda):=\sup\limits_{\mbox{\footnotesize finite set }W\subset J}\inf\limits_{\substack{k,n\in J\backslash W\\ k\neq n}}|\lambda_k-\lambda_n|>0.\label{eqn-weak-gap}
\end{equation}
\end{theorem}
The quantity $\gamma'(\Lambda)$ defined by \eqref{eqn-weak-gap} is called the \textit{weak uniform gap} of $\Lambda$.

\begin{remark}
Theorem~\ref{thm-ingham-2} is equivalent to say that: Given any uniformly separated $\Lambda $ and any finite subset $W\subset \Lambda$, if $|I|>\frac{2\pi}{\gamma\bigl(\Lambda\backslash W\bigr)}$, then the inverse inequality \eqref{eqn-inverse-inequality} holds. If there exists a sequence of finite subsets $\bigl\lbrace W_n \bigr\rbrace_{n\in \N}$  such that
\begin{equation}
    \gamma\bigl(\Lambda\backslash W_n\bigr)\to\infty, \quad n\to\infty,
\end{equation}
then the inverse inequality \eqref{eqn-inverse-inequality} holds  for any interval of length $|I|>0$. For example, let $\Lambda=\bigl\lbrace k^3:k\in\Z \bigr\rbrace$ and $W_n=\bigl\lbrace k^3:|k|\le n\bigr\rbrace$ for all $n\in \N$, we have $\gamma(\Lambda \backslash W_n)\rightarrow \infty$ when $n\to\infty$.
\end{remark}

For {further} generalizations of Ingham's Theorem, we refer the reader to \cite{ball1979nonharmonic} and \cite{komornik2005fourier}.
{The above results are the main tools to prove Proposition \ref{thm:main1}. In the sequel, we present a stronger proposition here in which (ii) and (iii) imply Proposition \ref{thm:main1}.}

\begin{proposition}\label{thm-sec2-main-1}
 Let $(t_0,x_0)\in \R\times \T, v\in\R$ and $\gamma'(v)>0$.
\begin{enumerate}
    \item[\rm{(i)}]  If {$\Pi(v)$ is finite}, then for every $T>0$, the inequality
    \begin{equation}
        \int_0^T|u(t_0+t,x_0-vt)|^2\d t\lesssim \int_\T |u_0(x)|^2\d x\label{eqn-thm-1-1}
    \end{equation}
    holds for all solutions of \eqref{eqn-abstract}.
    \item[\rm{(ii)}] If { $\Pi(v)$ is empty} and $T>\frac{2\pi}{\gamma'(v)}$, then the observability inequality
    \begin{equation}\label{thm-sec2-main-1-conlclusion-2}
    \int_\T|u_0(x)|^2\d x\lesssim \int_0^T|u(t_0+t,x_0-vt)|^2\d t
    \end{equation}
    holds for all solutions of \eqref{eqn-abstract}.
    \item [\rm{(iii)}]If {$\Pi(v)$ is non-empty}, then there exists a non-trivial solution satisfying
    \begin{equation*}
        u(t_0+t,x_0-vt)=0, \quad \forall t\in \R,
    \end{equation*}
    and thus the qualitative observability \eqref{eqn-qualitative} fails.
    \item[\rm{(iv)}]  For every $T>\frac{2\pi}{\gamma'(v)}$, the inequality
    \begin{equation}
        \sum_{k\in\Z} \frac{1}{n_k(v)}|d_k|^2\lesssim \int_0^T |u(t_0+t,x_0-vt)|^2\d t\label{eqn-dk}
    \end{equation}
    holds for all solutions of \eqref{eqn-abstract}, where
    \begin{equation}
    d_k=\sum_{m\in \Xi_k(v)}e^{i(p(m)t_0+mx_0)}c_m,\quad \forall k\in \Z.\label{eqn-thm-1-2}
    \end{equation}
\end{enumerate}
\end{proposition}

\begin{proof}
 (i) By the assumption $\gamma'(v)>0$ and $|\Pi(v)|<\infty$, we use the equivalent definition \eqref{equ-gamma(v)-2} of $\gamma'(v)$ to see that, the set $\{\lambda_k(v)\}_{k\in \Z \backslash W}$ is uniformly separated for some finite set $W$. Rewrite \eqref{eqn-one-point-representation} as
\begin{equation}\label{equ-proof-thm1-1}
    u(t_0+t,x_0-vt)=\sum_{k\in W} c_ke^{i(p(k)t_0+kx_0)} e^{i\lambda_k(v)t}+\sum_{k\in \Z \backslash W} c_ke^{i(p(k)t_0+kx_0)} e^{i\lambda_k(v)t}.
\end{equation}
Squaring both sides of \eqref{equ-proof-thm1-1} and applying the elementary inequality $|\sum_{1\leq k \leq m}a_k|^2\leq m\sum_{1\leq k\leq m}|a_k|^2$, we infer that
\begin{multline}\label{equ-proof-thm1-2}
    \int_0^T|u(t_0+t,x_0-vt)|^2\d t \\ \leq (|W|+1) \biggl(\int_0^T \sum_{k\in W}|c_k|^2 \d t + \int_0^T \bigl|\!\sum_{k\in \Z \backslash W} c_ke^{i(p(k)t_0+kx_0)} e^{i\lambda_k(v)t}\bigr|^2 \d t\biggr).
\end{multline}
Moreover, by Theorem~\ref{thm-ingham-1}, we have the direct inequality
$$
\int_0^T \bigl|\sum_{k\in \Z \backslash W} c_ke^{i(p(k)t_0+kx_0)} e^{i\lambda_k(v)t}\bigr|^2 \d t\lesssim \sum_{k\in \Z \backslash W}|c_k|^2.
$$
This, together with \eqref{equ-proof-thm1-2}, gives
$$
\int_0^T|u(t_0+t,x_0-vt)|^2\d t \lesssim \sum_{k\in \Z }|c_k|^2.
$$
Thus we obtain \eqref{eqn-thm-1-1} since $\|u_0\|^2_{L^2(\T)}=\sum_{k\in \Z }|c_k|^2.$
%

(ii) Since $|\Pi(v)|=0$ and $\gamma'(v)>0$, the set $\bigl\lbrace \lambda_{k}(v) \bigr\rbrace_{k\in\Z}$ is uniformly separated. The desired inequality Theorem~\ref{thm-sec2-main-1-conlclusion-2} follows directly from Theorem~\ref{thm-ingham-2}.

(iii) If $|\Pi(v)|\neq 0$, then $\Pi (v) \neq \O$. By definition, there exist two integers $k_1,k_2\in \Pi(v),k_1\neq k_2$ such that  $\lambda_{k_1}(v)=\lambda_{k_2}(v)$. Set the initial data
$$
u_0(x)=e^{-i\left(p(k_1)t_0+k_1x_0\right)}e^{ik_1x}-e^{-i\left(p(k_2)t_0 +k_2x_0\right)}e^{ik_2x}.
$$
Then the corresponding solution of \eqref{eqn-abstract} is given by
\begin{equation*}
    u(t,x)=e^{-i\left(p(k_1)t_0+k_1x_0\right)}e^{ip(k_1)t}e^{ik_1x}-e^{-i\left(p(k_2)t_0 +k_2x_0\right)}e^{ip(k_2)t}e^{ik_2x}.
\end{equation*}
Clearly, $u(t,x)$ is not identical to $0$, but
\begin{equation*}
    u(t_0+t,x_0-vt)=0,\quad\forall t\in \R.
\end{equation*}

(iv) By the definition \eqref{eqn-thm-1-2}, we can rewrite \eqref{eqn-one-point-representation} as
\begin{equation*}
    u(t_0+t,x_0-vt)=\sum_{k\in\Z} \frac{1}{n_k(v)}d_k e^{i\lambda_k(v)t}=\sum_{k\in \lbrace \lambda_j(v):j\in \Z \rbrace} d_k' e^{ikt}
\end{equation*}
where $d_k'=d_j$ for some $j$ who satisfies $\lambda_j(v)= k$.
The set $\lbrace \lambda_j(v):j\in\Z \rbrace$ is uniformly separated, then we obtain the desired inequality \eqref{eqn-dk} by Theorem~\ref{thm-ingham-2} (note that in \eqref{eqn-dk} each $d'_k,k\in \bigl\lbrace \lambda_j(v):j\in\Z\bigr\rbrace $ is counted $n_k(v)$ times).
\end{proof}

\begin{remark}
(i) Under the  assumptions of Proposition~\ref{thm-sec2-main-1} (ii), we have in fact
\begin{equation*}
    \int_\T|u_0(x)|^2\d x \asymp \int_0^T|u(t_0+t,x_0-vt)|^2\d t,
\end{equation*}
see Theorem~\ref{thm-ingham-1} (i).  By the same reason, the quantitative observability in Theorem~\ref{thm:main2} is equivalent to
\begin{equation*}
     \int_\T|u_0(x)|^2\d x \asymp \int_0^T \sum_{i=1}^2 | u(t_i+t,x_i-v_it)| ^2 \d t.
\end{equation*}

(ii)We note that \eqref{eqn-dk} can be rewritten as the following
\begin{equation}
    \sum_{k\in \Pi(v)} \frac{1}{n_k(v)}|d_k|^2+\sum_{k\in \Z\backslash\Pi(v)}|c_k|^2\lesssim \int_0^T |u(t_0+t,x_0-vt)|^2 \d t.\label{eqn-dk-ck}
\end{equation}
In fact, if $k\notin \Pi(v)$, then the set $\Xi_k(v)$ has only one element, namely $k$, so in this case, by \eqref{eqn-thm-1-2}, $|d_k|=|c_k|$. This gives \eqref{eqn-dk-ck}.
\end{remark}


\section{Proof of Theorem~\ref{thm:main2}}\label{sec:3}

 Before the proof of Theorem~\ref{thm:main2}, it is necessary to analyze the graph $G(v_1,v_2)$ and obtain some properties of it. Here and below, the graph $G(v_1,v_2)$ is always the one defined in Subsection~\ref{subsec:12}. We assume that (H1) and (H2) hold.

\begin{definition}
A path or a cycle is called alternative if any two consecutive edges have different colors.
\end{definition}

\begin{lemma}\label{thm-path-equiv}
 If the graph $G(v_1,v_2)$ contains a $2$-colored path (cycle, respectively), then  $G(v_1,v_2)$ also contains  an alternative path (alternative cycle, respectively).
\end{lemma}
\begin{proof}
We first consider a $2$-colored infinite simple path $\mathcal{L}$. By the hypothesis (H1), there exists no infinite path in $\mathcal{L}$ that has edges with only one color. Thus there exists a sequences $\{k_{m_j}\}$  of edges where the vertices changes color, i.e., for $m_j\leq n\leq m_{j+1}-1$, the edges $(k_n,k_{n+1})$ all have the same color, say red while for $m_{j-1}\leq n\leq m_j-1$ and $m_{j+1}\leq n\leq m_{j+2}-1$, the edges $(k_n,k_{n+1})$ all have the opposite color, blue.  

Next observe that if the edges $(k_n,k_{n+1})$ for $m_j\leq n\leq m_{j+1}-1$ are all red then 
$$
    \lambda_{k_{m_{j}}}(v_1)=\lambda_{k_{m_j+1}}(v_1)=\cdots \lambda_{k_{m_{j+1}}}(v_1).
$$
Thus, by construction, there is a red colored edge $(k_{m_j},k_{m_{j+1}})$ in the graph and we may replace the piece of $\mathcal{L}$ given by $(k_{m_j},k_{m_{j}+1})\cdots (k_{m_{j+1}-1},k_{m_{j+1}})$ by the single red colored vertice $(k_{m_j},k_{m_{j+1}})$, see  Figure~\ref{fig:3}.
\begin{figure}[htbp]
    \centering
    \vspace{0.05\textwidth}
    \begin{tikzpicture}[scale=0.6]
        \SetVertexMath
        \GraphInit[vstyle=Simple]
        \tikzset{VertexStyle/.append style = {
minimum size = 3pt, inner sep =0pt}}
        \draw[color=Red,thick,opacity=0.6] (0,0)--(1.4,1)--(3.4,1.6);
        \draw[color=Red,thick,opacity=0.6] (5.4,1.6)--(7.4,1)--(8.8,0);
        \draw[color=Red,dashed,thick,opacity=0.6] (3.4,1.6)--(5.4,1.6);
        \draw[color=Blue,thick,opacity=0.6] (-0.7,-1)--(0,0);
        \draw[color=Blue,thick,opacity=0.6] (8.8,0)--(9.5,-1);
        \Vertex[x=0,y=0]{1}
        \Vertex[x=1.4,y=1]{2}
        \Vertex[x=3.4,y=1.6]{3}
        \Vertex[x=5.4,y=1.6]{4}
        \Vertex[x=7.4,y=1]{5}
        \Vertex[x=8.8,y=0]{6}
        \draw (0,0)node[below]{\tiny $k_{m_j}$};
        \draw (1.4,1)node[below]{\tiny $k_{m_{j}+1}$};
        \draw (3.4,1.6)node[below]{\tiny  $k_{m_j+2}$};
        \draw (5.4,1.6)node[below]{\tiny  $k_{m_{j+1}-2}$};
        \draw (7.4,1)node[below]{\tiny  $k_{m_{j+1}-1}$};
        \draw (8.8,0)node[below]{\tiny $k_{m_{j+1}}$};
        \draw (9.5,0.5)node[right]{$\Longrightarrow$};
    \end{tikzpicture}
    \begin{tikzpicture}[scale=0.6]
        \SetVertexMath
        \GraphInit[vstyle=Simple]
        \tikzset{VertexStyle/.append style = {
minimum size = 3pt, inner sep =0pt}}
        \draw[color=Red,thick,opacity=0.6] (0,0)--(8.8,0);
        \draw[color=Blue,thick,opacity=0.6] (-0.7,-1)--(0,0);
        \draw[color=Blue,thick,opacity=0.6] (8.8,0)--(9.5,-1);
        \Vertex[x=0,y=0]{1}
        \Vertex[x=8.8,y=0]{6}
        \draw (0,0)node[below]{\tiny $k_{m_j}$};
        \draw (8.8,0)node[below]{\tiny $k_{m_{j+1}}$};
        \draw (4,0)node[above]{\tiny $\lambda_{k_{m_j}}(v_1)=\lambda_{k_{m_{j+1}}}(v_1)$};
    \end{tikzpicture}
    \vspace{0.05\textwidth}
    \caption{Two consequtive vertices in the subsequence.}
    \label{fig:3}
\end{figure}
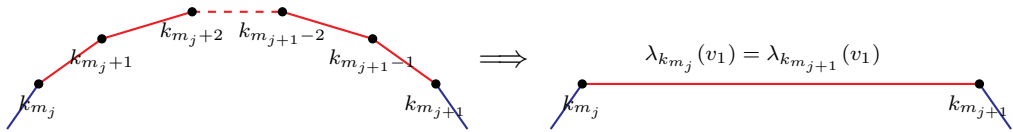
In the same way, we can replace the blue colored paths $(k_{m_{j-1}},k_{m_{j-1}+1})\cdots(k_{m_j-1},k_{m_j})$ and $(k_{m_{j+1}},k_{m_{j+1}+1})\cdots(k_{m_{j+2}-1},k_{m_{j+2}})$ by single blue colored vertices $(k_{m_{j-1}},k_{m_j})$ and $(k_{m_{j+1}},k_{m_{j+2}})$.
The resulting path is the one we are looking for.

The alternative cycle  can be constructed by the same way.
\end{proof}

 Given $(t_1,x_1),(t_2,x_2)\in \R\times \T$. Suppose  that the solution $u$ of \eqref{eqn-abstract} satisfies
 \begin{equation}
     u(t_i+t,x_i-v_it)=0,\quad \forall t\in (0,T),i=1,2, T>\frac{2\pi}{\min \lbrace\gamma'(v_1),\gamma'(v_2)\rbrace}.\label{eqn-main-thm-assump}
 \end{equation}
 Then the qualitative observability is equivalent to show $u_0 \equiv 0$, or (by \eqref{equ-data}) the Fourier coefficients $c_k=0$ for all $k\in \Z$. We present a simple criterion on $k$ to ensure $c_k=0$.  Let $k$ be a vertex in $G(v_1,v_2)$ which has only edges of one color, say red without loss of generality. Then $k\notin \Pi(v_2)$. By {Proposition}~\ref{thm-sec2-main-1} (iv) (see \eqref{eqn-dk} and \eqref{eqn-dk-ck}) and the assumption \eqref{eqn-main-thm-assump}, we obtain $c_k=0$.

The above illustration implies that for any vertex who has only edges of one color, $c_k=0$ under the assumption \eqref{eqn-main-thm-assump}. Hence we can remove these vertices and edges that are incident on them when we prove the qualitative observability \eqref{eqn-qualitative}. Then we obtain a new graph. We can remove vertices and edges from the new graph by the same operation. For a finite graph, we keep repeating the above operation again and again until it {terminates, which will happen in finite steps.} If the graph is not finite, the above operation may not terminate. To remove all these ``useless'' vertices at the same time, we need a definition of them without iteration.

\begin{definition}
Given any vertex $k$ in $G(v_1,v_2)$, a vertex $k'$ is said to be connected to $k$ in red if there exists a finite path whose initial vertex is $k'$ and terminal vertex is $k$, and the last edge is red-colored. An edge $e$ is said to be connected to $k$ in red if there exists a path whose initial vertex is an end of $e$ and terminal vertex is $k$, and the last edge is red-colored. We also definie that $k$ is connected to itself in red(see Figure~\ref{fig:4} for examples).

All these vertices and edges connected to $k$ in red form a subgraph, called the red-portion of $k$ in $G(v_1,v_2)$. The blue-portion can be defined symmetrically.

\end{definition}
\begin{figure}[htbp]
    \centering
    \begin{tikzpicture}[scale=0.5]
        \SetVertexMath
        \GraphInit[vstyle=Simple]
        \tikzset{VertexStyle/.append style = {
minimum size = 3pt, inner sep =0pt}}
        \Vertex[x=2,y=1.732]{1}
        \Vertex[x=0,y=1.732]{2}
        \draw[Blue] (1)--(2);
        \draw (1) node[below]{\tiny $k$};
        \Vertex[x=10,y=1.732]{1r}
        \draw (1r) node[below]{\tiny $k$};
        \Vertex[x=0,y=-4]{333}
        \Vertex[x=-1,y=-2.268]{444}
        \Vertex[x=0,y=-0.536]{555}
        \Vertex[x=2,y=-0.536]{666}
        \Vertex[x=3,y=-2.268]{111}
        \Vertex[x=2,y=-4]{222}
        \draw[Red,thick,opacity=0.6] (111)--(222);
        \draw[Blue,thick,opacity=0.6] (222)--(333);
        \draw[Red,thick,opacity=0.6] (333)--(444);
        \draw[Blue,thick,opacity=0.6] (444)--(555);
        \draw[Red,thick,opacity=0.6] (555)--(666);
        \draw[Blue,thick,opacity=0.6] (666)--(111);
        \draw (111) node[below]{\tiny $k$};
        \Vertex[x=0+9,y=-4]{333r}
        \Vertex[x=-1+9,y=-2.268]{444r}
        \Vertex[x=0+9,y=-0.536]{555r}
        \Vertex[x=2+9,y=-0.536]{666r}
        \Vertex[x=3+9,y=-2.268]{111r}
        \Vertex[x=2+9,y=-4]{222r}
        \draw (111r) node[below]{\tiny $k$};
        \draw[Red,thick,opacity=0.6] (111r)--(222r);
        \draw[Blue,thick,opacity=0.6] (222r)--(333r);
        \draw[Red,thick,opacity=0.6] (333r)--(444r);
        \draw[Blue,thick,opacity=0.6] (444r)--(555r);
        \draw[Red,thick,opacity=0.6] (555r)--(666r);
        \draw[Blue,thick,opacity=0.6] (666r)--(111r);
        \Vertex[x=0,y=-4-4]{33}
        \Vertex[x=-1,y=-2.268-4]{44}
        \Vertex[x=0,y=-0.536-4]{55}
        \Vertex[x=2,y=-0.536-4]{66}
        \Vertex[x=3,y=-2.268-4]{11}
        \Vertex[x=2,y=-4-4]{22}
        \Vertex[x=5,y=-2.268-4]{88}
        \draw[Red,thick,opacity=0.6] (11)--(22)--(55)--(66)--(11)--(55) (22)--(66);
        \draw[Blue,thick,opacity=0.6] (22)--(33);
        \draw[Red,thick,opacity=0.6] (33)--(44);
        \draw[Blue,thick,opacity=0.6] (44)--(55);
        \draw[Blue,thick,opacity=0.6] (11)--(88);
        \draw (11) node[below]{\tiny $k$};
        \Vertex[x=0+9,y=-4-4]{33r}
        \Vertex[x=-1+9,y=-2.268-4]{44r}
        \Vertex[x=0+9,y=-0.536-4]{55r}
        \Vertex[x=2+9,y=-0.536-4]{66r}
        \Vertex[x=3+9,y=-2.268-4]{11r}
        \Vertex[x=2+9,y=-4-4]{22r}
        \draw[Red,thick,opacity=0.6] (11r)--(22r)--(55r)--(66r)--(11r)--(55r) (22r)--(66r);
        \draw[Blue,thick,opacity=0.6] (22r)--(33r);
        \draw[Red,thick,opacity=0.6] (33r)--(44r);
        \draw[Blue,thick,opacity=0.6] (44r)--(55r);
        \draw (11r) node[below]{\tiny $k$};
        \draw (5,1.732) node[right]{$\Longrightarrow$};
        \draw (5,-2.268) node[right]{$\Longrightarrow$};
        \draw (5,-6) node[right]{$\Longrightarrow$};
        \draw (2,-9) node[below]{\small $ G(v_1,v_2)$};
        \draw (10,-9) node[below]{The red-portion of $k$};
        \draw (-1-0.5,1.732) node[left]{\small (a)};
        \draw (-1-0.5,-2.268) node[left]{\small (b)};
        \draw (-1-0.5,-6.3) node[left]{\small (c)};
    \end{tikzpicture}
    \caption{Examples of the red-portion of $k$ in $G(v_1,v_2)$}
    \label{fig:4}
\end{figure}
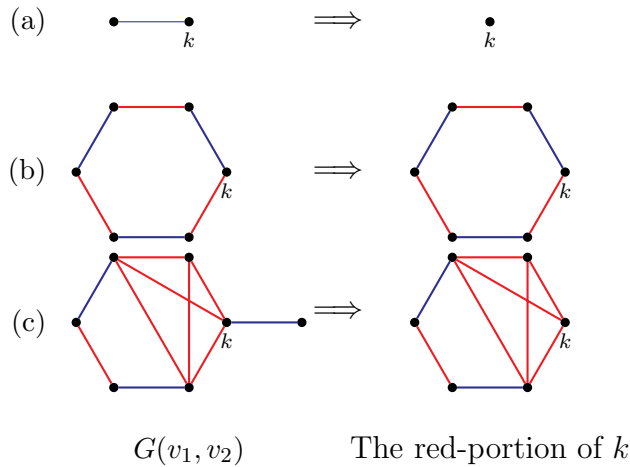

\begin{lemma}\label{thm-reduced-lemma}
 Let $k$ be a vertex in $G(v_1,v_2)$ { such that either}
\begin{enumerate}
    \item [\rm{(i)}] The red-portion of $k$ contains no $2$-colored cycle and has finite vertices, or
    \item [\rm{(ii)}]The blue-portion of $k$ contains no $2$-colored cycle and has finite vertices.
\end{enumerate}
Then the assumption \eqref{eqn-main-thm-assump} implies $c_k=0$.

Moreover, every vertex in a red-portion or blue-portion which contains no $2$-colored cycle and has finite vertices satisfies {\rm{(i)}} or {\rm{(ii)}}.
\end{lemma}

\begin{proof}
By the symmetric property, it suffices to show that, under the condition (i),  the assumption \eqref{eqn-main-thm-assump} implies $c_k=0$. We prove it by induction. Let $n$ be the number of all vertices of the red-portion who satisfies (i). For $n=1$, then there exists only one vertex $k$, this implies that $k$ only has blue-colored edges incident with it, then $k\notin \Pi(v_1)$ and $c_k=0$ by Proposition~\ref{thm-sec2-main-1} (iv). For $n=2$, there exists only one neighborhood $k_1$ of $k$ in the red-portion. If $k_1$ has another neighborhood $k_2\neq k$, then $k_2$ must be connected to $k$ along a path through a red-colored edge incident with $k$, which contradicts to $n=2$. Hence $k_1$ has only one neighborhood, this implies $c_{k_1}=0$ by Proposition~\ref{thm-sec2-main-1} (iv). Still by Proposition~\ref{thm-sec2-main-1} (iv), we have 
\begin{equation*}
    e^{i({p(k)}t_1+kx_1)}c_k+e^{i{p(k_1)}t_1+k_1x_1)}c_{k_1}=0.
\end{equation*}
This implies $c_k=0$. Now we assume that the lemma under condition (i) is true for the number of vertices of the red-portion of $k$ smaller than $n$, and consider the case of $n$ vertices. If the red-portion has a vertex with only one-colored edges, then we can remove it and reduce to the case of $n-1$ vertices. Indeed, the red-portion who satisfies condition (i) always has such vertices. Otherwise, starting from a neighborhood $k_1$ of $k$ in the red-portion, and pass to a neighborhood $k_2$ of $k_1$ with the blue-colored edge $e(k_1,k_2)$. Since $k_2$ has both colored edges, we choose a neighborhood $k_3$ of $k_2$ with the red-colored edge $e(k_2,k_3)$. Repeat this approach and ensure $k_i\notin \bigl\lbrace k_1,\cdots,k_{i-1}\bigr\rbrace$ for each step $i$. Since the number of vertices is finite, it will terminate when $k_i=k_{j}$ for some $j<i$. This implies that the red-portion has a two-colored cycle, which contradicts to the condition (i).

Now we prove the second statement. Without loss of generality, let $k'$ be any vertex in the red-portion of $k$ who contains no $2$-colored cycle and has finite vertices, suppose that it does not satisfy condition (i) and (ii). Then both the red-portion and blue-portion of $k'$ contain a $2$-colored cycle or infinite {path}. However, either $2$-colored cycle or infinite vertices is contained in the red-portion of $k$, which leads to a contradiction.
\end{proof}

Now we remove all vertices, in the graph $G(v_1,v_2)$, satisfying condition (i) or (ii) in Lemma~\ref{thm-reduced-lemma}, and obtain a new graph. We denote it by $\Tilde{G}(v_1,v_2)$, called the \textit{reduced graph} of $G(v_1,v_2)$, see Figure~\ref{fig:5}.
\begin{figure}[htbp]
    \centering
     \begin{tikzpicture}
        \SetVertexMath
        \GraphInit[vstyle=Simple]
        \tikzset{VertexStyle/.append style = {
minimum size = 3pt, inner sep =0pt}}
        \Vertex[x=1,y=1]{0}
        \Vertex[x=0,y=0]{1}
        \Vertex[x=2,y=0]{2}
        \Vertex[x=1,y=-1]{3}
        \Vertex[x=1,y=-2.2]{4}
        \Vertex[x=2.5,y=0.6]{5}
        \Vertex[x=4.1,y=0.6]{6}
        \Vertex[x=4.1,y=-1]{7}
        \Vertex[x=2.5,y=-1]{8}
        \Vertex[x=5,y=-1]{9}
        \Vertex[x=6.6,y=-1]{10}
        \Vertex[x=5.8,y=0.2]{11}
        \draw[Red,thick,opacity=0.6] (0)--(1);
        \draw[Blue,thick,opacity=0.6] (1)--(3);
        \draw[Red,thick,opacity=0.6] (3)--(2);
        \draw[Blue,thick,opacity=0.6] (2)--(0);
        \draw[Red,thick,opacity=0.6] (3)--(4);
        \draw[Red,thick,opacity=0.6] (2)--(4);
        \draw[Red,thick,opacity=0.6] (5)--(6);
        \draw[Blue,thick,opacity=0.6] (6)--(7);
        \draw[Red,thick,opacity=0.6] (7)--(8);
        \draw[Blue,thick,opacity=0.6] (8)--(5);
        \draw[Red,thick,opacity=0.6] (9)--(10)--(11)--(9);
        \draw (7,0) node[right] {\Large $\Longrightarrow$};
        \Vertex[x=10,y=1]{0r}
        \Vertex[x=9,y=0]{1r}
        \Vertex[x=11,y=0]{2r}
        \Vertex[x=10,y=-1]{3r}
        \draw[Red,thick,opacity=0.6] (0r)--(1r);
        \draw[Blue,thick,opacity=0.6] (1r)--(3r);
        \draw[Red,thick,opacity=0.6] (3r)--(2r);
        \draw[Blue,thick,opacity=0.6] (0r)--(2r);
        \Vertex[x=11.5,y=0.6]{5r}
        \Vertex[x=13.1,y=0.6]{6r}
        \Vertex[x=13.1,y=-1]{7r}
        \Vertex[x=11.5,y=-1]{8r}
        \draw[Red] (5r)--(6r);
        \draw[Blue,thick,opacity=0.6] (6r)--(7r);
        \draw[Red,thick,opacity=0.6] (7r)--(8r);
        \draw[Blue,thick,opacity=0.6] (8r)--(5r);
        \draw (2,-1.7) node[right] {$G(v_1,v_2)$};
        \draw (10,-1.7) node[right] {$\Tilde{G}(v_1,v_2)$};
     \end{tikzpicture}
    \caption{Reduction of the graph $G(v_1,v_2)$}
    \label{fig:5}
\end{figure}
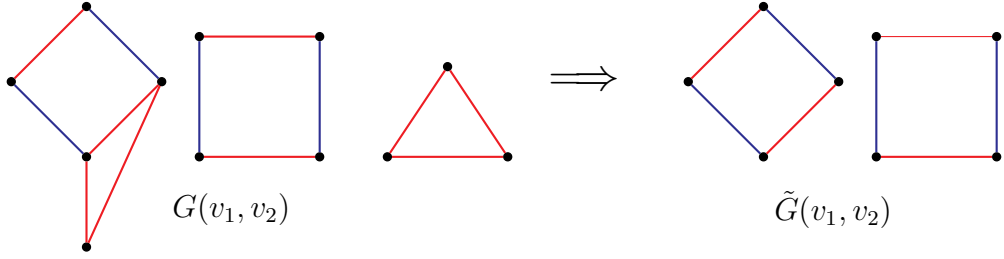

In order to use it, we need some lemmas to explain the equivalence of $G(v_1,v_2)$ and $\Tilde{G}(v_1,v_2)$ in some sense.

\begin{lemma}\label{thm-lemma-1}
Every vertex of the reduced graph $\Tilde{G}(v_1,v_2)$ has both red-colored and blue-colored edges.
\end{lemma}
\begin{proof}
We prove it by contradiction. Assume that the vertex $k$ in $\Tilde{G}(v_1,v_2)$ has only red-colored edges. Then it must have both colored edges in $G(v_1,v_2)$ due to the construction of $\Tilde{G}(v_1,v_2)$ and Lemma~\ref{thm-reduced-lemma}. Since $k\in \Tilde{G}(v_1,v_2)$, it does not satisfy condition (i) and (ii) in Lemma~\ref{thm-reduced-lemma}.
Let $k_1\in G(v_1,v_2)$ be any neighborhood of $k$ with a blue-colored edge. Since $k$ does not satisfies (i), and the red-portion of $k$ is a subgraph of the blue-portion of $k_1$, according to the second statement of Lemma~\ref{thm-reduced-lemma}, we see that $k_1$ does not satisfy (ii). Since $k_1\notin \Tilde{G}(v_1,v_2)$ (this follows from our hypothesis on $k$ and the choices of $k_1$), it must satisfy (i). However, the red-portion of $k_1$ is a subgraph of the blue-portion of $k$ and this implies that $k$ satisfies (ii), which leads to a contradiction.
\end{proof}

By {Lemma~\ref{thm-path-equiv} and Lemma~\ref{thm-lemma-1}} , we have the following lemma directly:
\begin{lemma}\label{thm-lemma-2}
 The reduced graph $\Tilde{G}(v_1,v_2)$ contains a two colored two-way infinite path $($two-colored cycle, respectively$)$ if and only if  the graph $G(v_1,v_2)$ contains a two colored two-way  infinite path $($two-colored cycle, respectively$)$.
\end{lemma}

Now we are ready to prove Theorem~\ref{thm:main2}.

\medskip

\noindent{\itshape \bf Proof of Theorem \rm{\ref{thm:main2}} (i)}.

{\bf ``If part''.}
Suppose that $G(v_1,v_2)$ has no two-colored cycles. By Lemma~\ref{thm-reduced-lemma} and the assumption \eqref{eqn-main-thm-assump}, we deduce $c_k=0$ for each vertex $k$ in $G(v_1,v_2)$ but not in $\Tilde{G}(v_1,v_2)$. Hence it remains to consider vertices in $\Tilde{G}(v_1,v_2)$. We consider the first case that $G(v_1,v_2)$ has neither two-colored cycle nor one-way infinite path. By Lemma~\ref{thm-lemma-2}, the graph $\Tilde{G}(v_1,v_2)$ contains neither two-way infinite path nor two-colored cycle. This combines with Lemma~\ref{thm-lemma-1} implies that the graph  $\Tilde{G}(v_1,v_2)$ is empty. Hence $c_k=0$ for all $k\in \Z$, and the desired qualitative observability holds. Now we assume that $\Tilde{G}(v_1,v_2)$ has no two-colored cycle but has a one-way infinite path. We call a one-way infinite path $\mathcal{L}:=\bigl\lbrace k_i,e:i\in\N,e=e(k_i,k_{i+1}),\forall i\in \N \bigr\rbrace$ is primitive if the number of neighborhoods of each $k_i,i\in\N$ in $\Tilde{G}(v_1,v_2)$ is equal to $2$. By Proposition~\ref{thm-sec2-main-1} (iv)
we obtain
\begin{equation}
    e^{i\omega_i}c_{k_i}+e^{i\omega_{i+1}}c_{k_{i+1}}=0\label{eqn-fix}
\end{equation}
where $\omega_i,\omega_{i+1}\in \R,i\in\N$. Clearly, \eqref{eqn-fix} implies that $|c_{k_i}|=|c_{k_{i+1}}|$. Recall that $\|u_0\|_{L^2(\T)}^2=2\pi\sum_{k\in\Z}|c_k|^2<\infty$, we conclude that $c_{k_i}=0$ for all vertices $k_{i}$ in a primitive one-way infinite path.
Hence we can remove all vertices belonging to a primitive one-way infinite paths, and get a new graph $H(v_1,v_2)$. The new graph $H(v_1,v_2)$ has no one-way infinite path. Indeed, if it has a one-way infinite path $\mathcal{L}'$, by the hypothesis (H2), then we can find a primitive one-way infinite path contained in $\mathcal{L}'$, which is impossible. Hence the graph $H(v_1,v_2)$ has only finitely many vertices and no $2$-colored cycle, then using Lemma~\ref{thm-reduced-lemma} we obtain $c_k=0$ for all $k$ in $H(v_1,v_2)$.

{\bf ``Only if part''.} To prove the converse, we argue by contradiction. Suppose that $G(v_1,v_2)$ has a two-colored cycle, we shall construct a non-trivial solution satisfying \eqref{eqn-main-thm-assump}. In fact, by Lemma~\ref{thm-path-equiv} there exists an alternative cycle whose vertices are arranged in a cyclic sequence $\bigl\lbrace k_j \bigr\rbrace_{j=1}^{2m}$ with red-colored edges $e(k_{2j-1},k_{2j})$ and blue-colored edges $e(k_{2j},k_{2j+1})$, $j=1,2,\cdots,m$ and identify $k_{2m+1}$ with $k_{1}$. This is possible since the number of vertices in a two-colored alternative cycle is even.

 We first note that the system of equations
\begin{equation}\label{equ-ucp-diff-start-1}
    \begin{cases}
        x=x_1-v_1(t-t_1) &  \\
        x=x_2-v_2(t-t_2) &
    \end{cases}
\end{equation}
has a unique solution $(t,x)=(t_0,x_0)$ with
\begin{equation}
    \begin{cases}
        t_0= \frac{(x_1+v_1t_1)-(x_2+v_2t_2)}{v_1-v_2} &  \\
        x_0= \frac{v_1(x_2+v_2t_2)-v_2(x_1+v_1t_1)}{v_1-v_2}
    \end{cases}\label{equ-ucp-diff-start-2}.
\end{equation}
We set
\begin{equation}
    c_{k_j}=\begin{cases}
        e^{-i(p(k_j)t_0+k_jx_0)}, & 2\nmid j, \\
        -e^{-i(p(k_j)t_0+k_jx_0)}, & 2\mid  j,\label{eqn-equal-start}
    \end{cases}
\end{equation}
for $j=1,2,\cdots, 2m$ and $c_k=0$ for $k\notin \bigl\lbrace k_1,k_2,\cdots,k_{2m} \bigr\rbrace$.
Consider the solution $u(t,x)$ of \eqref{eqn-abstract} with the initial data $u_0=\sum\limits_{j=1}^{2m}c_{k_j}e^{i k_jx}$. Then we have
\begin{equation}\label{equ-ucp-diff-start-3}
    \begin{array}{ll}
       &u(t_1+t,x_1-v_1t)\\
   &=\sum_{j=1}^{2m}(-1)^{j-1}e^{-i\left(p(k_j)t_0+k_jx_0\right)}e^{i\left((p(k_j)t_1+k_j x_1\right)}e^{i\lambda_{k_j}(v_1)t}  \\
   &=\sum_{j=1}^{m}\bigl( e^{i\left(p(k_{2j-1})t_1+k_{2j-1}x_1-p(k_{2j-1})t_0-k_{2j-1}x_0)\right)}\\
    & \qquad \quad - e^{i\left(p(k_{2j})t_1+k_{2j}x_1-p(k_{2j})t_0-k_{2j}x_0)\right)} \bigr)e^{i\lambda _{k_{2j-1}}(v_1)t},\quad \forall t\in\R,
    \end{array}
\end{equation}
where in the last identity we used
\begin{equation}\label{equ-ucp-diff-start-4}
    \lambda_{k_{2j-1}}(v_1)=\lambda_{k_{2j}}(v_1),\quad j=1,2,\cdots,
\end{equation}
which follows from the fact that the vertices $k_{2j-1}$ and $k_{2j}$ are adjacent through red color edges.

Also, since $(t_0,x_0)$ given by \eqref{equ-ucp-diff-start-2} solves \eqref{equ-ucp-diff-start-1}, we have
$$
  x_0-x_1 =-v_1(t_0-t_1),  \quad x_0-x_2=-v_2(t_0-t_2).
$$
Then we deduce from \eqref{equ-ucp-diff-start-3} that
\begin{align}\label{equ-linevanish-1}
   &u(t_1+t,x_1-v_1t) \nonumber\\
    &=\sum_{j=1}^{m}\bigl( e^{i\left(p(k_{2j-1})-k_{2j-1}v_1\right)(t_1-t_0)}
         - e^{i\left(p(k_{2j})-k_{2j}v_1\right)(t_1-t_0)} \bigr)e^{i\lambda _{k_{2j-1}}(v_1)t}\nonumber \\
    & =0, \quad \forall t\in\R.
    \end{align}
Here we used that the coefficients of  $e^{i\lambda _{k_{2j-1}}(v_1)t}$ are zero, which follows from \eqref{equ-ucp-diff-start-4} and the definition of $\lambda_k(v)$.

Similarly, since the vertex $k_{2m+1}$ is identified with $k_1$, we deduce  that
\begin{equation*}
    u(t_2+t,x_2-v_2t)=0,\quad \forall t\in\R.
\end{equation*}
Hence there exists a nontrivial solution satisfies \eqref{eqn-main-thm-assump}.

\medskip

\noindent{\itshape \bf Proof of Theorem \rm{\ref{thm:main2}} (ii)}.

{\bf ``If part''.}
  We assume that the graph $G(v_1,v_2)$ has no two-colored cycle and $g(v_1,v_2)<\infty$. We split the discussion into two cases.

 {\bf Case (1):}
the graph $G(v_1,v_2)$ has only one connected component. Then the number of elements in $\Pi(v_1)\cup\Pi(v_2)$ satisfies
\begin{align}\label{equ-proof-thm1.2-ii-1}
|\Pi(v_1)\cup\Pi(v_2)|\leq g(v_1,v_2)<\infty.
\end{align}
Define
\begin{equation}\label{equ-proof-thm1.2-ii-2}
    d_{i,k}:=\sum_{m\in\Xi_{k}(v_i)} e^{i(p(m)t_0+mx_0)}c_m,\quad \forall k\in\Z,i=1,2.
\end{equation}
By Proposition~\ref{thm-sec2-main-1} (iv) we have for $i=1,2$
\begin{equation}\label{equ-proof-thm1.2-ii-3}
    \sum_{k\in \Pi(v_i)}\frac{1}{n_k(v_i)}|d_{i,k}|^2+\sum_{k\in\Z\backslash \Pi(v_i)}|c_k|^2\lesssim \int_0^T|u(t_i+t,x_i-v_it)|^2\d t.
\end{equation}
By the assumption (H1), both $n_k(v_1)$ and $n_k(v_2)$ are finite for every $k\in \Z$. Recall the fact \eqref{equ-proof-thm1.2-ii-1}, we infer
\begin{equation}\label{equ-proof-thm1.2-ii-4}
\max\bigl\lbrace n_k(v_i): k\in \Pi(v_1)\cup\Pi(v_2),i=1,2\bigr\rbrace<\infty.
\end{equation}
Summing \eqref{equ-proof-thm1.2-ii-3} for $i=1,2$ and using \eqref{equ-proof-thm1.2-ii-4}, we find
\begin{equation*}
    \sum_{k\in\Pi(v_1)\cup \Pi(v_2)}\bigl(|d_{1,k}|^2+|d_{2,k}|^2\bigr)+\sum_{k\in \Z \backslash\bigl(\Pi(v_1)\cup\Pi(v_2)\bigr)}|c_k|^2\lesssim \int_0^T\sum_{i=1}^2 |u(t_i+t,x_i-v_it)|^2\d t.
\end{equation*}
Hence the quantitative observability holds if one can show that
\begin{equation}
    \sum_{k\in \Pi(v_1)\cup\Pi(v_2)}|c_k|^2\lesssim \sum_{k\in\Pi(v_1)\cup\Pi(v_2)}\bigl(|d_{1,k}|^2+|d_{2,k}|^2\bigr).\label{eqn-sufficient-inequality}
\end{equation}
We prove \eqref{eqn-sufficient-inequality} by contradiction. Assume that there exist sequences $\bigl\lbrace c_k^{(j)}\bigr\rbrace_{j\in \N}$  for all $k\in \Pi(v_1)\cup\Pi(v_2)$ such that
\begin{equation}
    \sum\limits_{k\in\Pi(v_1)\cup\Pi(v_2)}\bigl|c_k^{(j)}\bigr|^2=1,\quad \forall j\in\N,\label{eqn-the-wlg}
\end{equation}
and when $j\to \infty $
\begin{equation}
    \sum\limits_{k\in\Pi(v_1)\cup\Pi(v_2)}\bigl(\bigl|d^{(j)}_{1,k}\bigr|^2+\bigl|d_{2,k}^{(j)}\bigr|^2\bigr)
    \rightarrow 0,\label{eqn-contrad-limits}
\end{equation}
where $d_{i,k}^{(j)}=\sum\limits_{m\in \Xi_k(v_i)}e^{i(p(m)t_0+mx_0)}c_m^{(j)}$ for all $j\in \N,k\in\Pi(v_1)\cup\Pi(v_2), i=1,2$.

Since $|\Pi(v_1)\cup\Pi(v_2)|$ is finite, there exists a subsequence of $c_k^{(j)}$, also denoted by $c_k^{(j)}$,  such that
\begin{equation}\label{equ-proof-thm1.2-ii-5}
c_k^j \to c_k^0, \quad \forall k\in \Pi(v_1)\cup\Pi(v_2).
\end{equation}
On the other hand, by \eqref{eqn-contrad-limits} we obtain when $j\rightarrow \infty$
$$
    d^{(j)}_{i,k}\rightarrow 0,\quad\forall k\in\Pi(v_1)\cup\Pi(v_2), i=1,2.
$$
Thus, by the qualitative observability in Theorem~\ref{thm:main2} (i), the limit in \eqref{equ-proof-thm1.2-ii-5} satisfies $c_k^{(0)}=0$ for all $k\in\Pi(v_1)\cup\Pi(v_2)$. But this contradicts to \eqref{eqn-the-wlg}.

  {\bf Case (2):} the graph $G(v_1,v_2)$ has more than $1$ connected components. Assume the graph has $m$ components and denote all components of the graph by $H_1,H_2,\cdots,H_m$. According to the proof of Case (1), for each connected component $H_i\subset G(v_1,v_2)$ with $|H_i|\le {M:= g(v_1,v_2)<\infty}, i=1,2,\cdots,m$, \eqref{eqn-sufficient-inequality} can be written as
\begin{equation}\label{eqn-for-each}
    \sum_{k\in H_i\bigcap (\Pi(v_1)\cup\Pi(v_2))}|c_k|^2 \le C(M,T)\sum_{k\in H_i\bigcap (\Pi(v_1)\cup\Pi(v_2))}\bigl(|d_{1,k}|^2+|d_{2,k}|^2\bigr),
\end{equation}
where $C(M,T)$ depends only on $M$ and $T$ { and uniform for the number of connected components} (Indeed, a connected component with at most $M$ vertices has only finitely many possible forms, hence we may choose the largest constant as $C(M,T)$). We sum \eqref{eqn-for-each} for $i=1,2,\cdots,m$ and get
\begin{equation}
    \sum_{k\in \Pi(v_1)\cup\Pi(v_2)}|c_k|^2\le C(M,T)\sum_{k\in\Pi(v_1)\cup\Pi(v_2)}\bigl( |d_{1,k}|^2+|d_{2,k}|^2 \bigr).
\end{equation}
If the graph $G(v_1,v_2)$ has infinitely many connected components, we can use the conclusion of finitely many components and take the limit.

{\bf ``Only if part''.}
Assume the quantitative observability \eqref{eqn-quantitative} holds. Then by the proof of Theorem~\ref{thm:main2} (i), the graph $G(v_1,v_2)$ has no two-colored cycle. If the graph $G(v_1,v_2)$ is finite, we certainly have $g(v_1,v_2)<\infty$. So we assume that the graph $G(v_1,v_2)$ is not finite now. To show $g(v_1,v_2)<\infty$, we argue by contradiction. To this end, we assume  $g(v_1,v_2)=\infty$, and split the discussion into two cases.

{\bf Case (1):} the graph $G(v_1,v_2)$ has a one-way infinite path. By Lemma~\ref{thm-lemma-1} there exists an alternative one-way infinite path whose vertices are arranged in a linear sequence $\bigl\lbrace k_i\bigr\rbrace \subset \N$ where the edge $e(k_{2i-1},k_{2i})$ is red-colored and $e(k_{2i},k_{2i+1})$ is blue-colored for all $i\in\N$. For every $n\in \N$, we set
    \begin{equation}
        u_{0,n}:=\sum_{j=1}^{2n}c_{k_j}e^{ik_j x}, \quad c_{k_j}=(-1)^je^{-i(p(k_j)t_0+k_jx_0)},
    \end{equation}
    where $(t_0,x_0)$ is defined by \eqref{equ-ucp-diff-start-2}. Then
    \begin{equation}\label{eqn-war-1}
        \int_\T|u_{0,n}(x)|^2\d x=2\pi\sum_{i=1}^{2n}| c_{k_i}|^2= {4\pi n.}
    \end{equation}
However, it follows the same way as \eqref{equ-linevanish-1}  that the corresponding solution $u_n$ satisfies
    \begin{equation}
        \int_0^T\sum_{i=1}^2|u_{n}(t_i+t,x_i-v_it)|^2\d t= \int_0^T |u_{n}(t_2+t,x_2-v_2t)|^2\d t\leq 4T,\label{eqn-war-2}
    \end{equation}
    where the only difference here is that the initial and terminal terms do not cancel out.
    Then by \eqref{eqn-war-1} and \eqref{eqn-war-2} we obtain
    \begin{equation}
       \frac{\displaystyle\int_0^T\sum\limits_{i=1}^2|u(t_i+t,x_i-v_it)|^2\d t}{\displaystyle\int_\T |u_{0,n}(x)|^2\d x}\to 0
    \end{equation}
    when $n\to \infty$.
    This contradicts to \eqref{eqn-quantitative}.

{\bf Case (2):} the graph $G(v_1,v_2)$ has no one-way infinite path and $g(v_1,v_2)=\infty$. Then we must have a sequence of finite path $\bigl\lbrace \mathcal{L}_n\bigr\rbrace_{n\in\N}$ such that for each $n\in\N$, $\mathcal{L}_n$ is equipped with a sequence of vertices $\lbrace k_{n,1},k_{n,2},\cdots,k_{n,m_n}\rbrace$ and
\begin{equation}\label{equ-sequence-mn-1}
    m_1<m_2<\cdots<m_n<\cdots.
\end{equation}
For every $n$, we set
\begin{equation*}
    u_{0,n}:=\sum_{j=1}^{m_n} c_{k_{n,j}}, \quad
    {c_{k_{n,j}}=(-1)^i e^{-i(p(k_{n,j})t_0+k_{n,j}x_0)},} 
\end{equation*}
 where $(t_0,x_0)$ is defined by \eqref{equ-ucp-diff-start-2}. Similar to Case (1), we have
\begin{equation}\label{equ-sequence-mn-2}
    \int_\T |u_{0,n}|^2\d x=2\pi m_n,
\end{equation}
and
\begin{equation}\label{equ-sequence-mn-3}
    \int_0^T\sum_{i=1}^2|u_{i}(t_i+t,x_i-v_it)|^2\d t\leq 4T.
\end{equation}
From \eqref{equ-sequence-mn-1}, $m_n\to \infty$ as $n\to \infty$. Letting $n\to \infty$, the bounds \eqref{equ-sequence-mn-2}-\eqref{equ-sequence-mn-3} lead a contradiction to \eqref{eqn-quantitative}.

So we conclude that $g(v_1,v_2)<\infty$, this completes the proof.

\section{Proof of Theorem~\ref{thm:main3}}
\label{sec:4}
Let $p$ be a real polynomial with degree $\mathrm{deg} \, p\geq 3$.
Then it is easy to see that the assumptions (H1) and (H2) hold. By Theorem  \ref{thm:main2}, Theorem~\ref{thm:main3} follows if one can show that
\begin{equation*}
    g(v_1,v_2)<\infty, \quad v_1\neq v_2\in \R
\end{equation*}
and
\begin{equation*}
    \gamma'(v)=\infty, \quad v\in \R.
\end{equation*}

We split the proof into two cases.

{\bf Case (1):} $\mathrm{deg}\, p$ is odd. Assume that
$$
p(x)=a_{2n+1}x^{2n+1}+\cdots+a_1x+a_0, \quad a_{2n+1}\neq 0, n\geq 1.
$$
For every $v\in \R$,  the relation $\lambda_k(v)=\lambda_m(v)$ is equivalent to
\begin{equation}\label{eqn-odd-1}
    p(k)-p(m)-v(k-m)=0.
\end{equation}
Since the roots of $x^{2n+1}=1$ are $e^{i\frac{2\pi j}{2n+1}},j=1,\cdots,2n+1$, we have the factorization
\begin{equation}\label{eqn-factorization}
k^{2n+1}-m^{2n+1}=(k-m)\prod_{j=1}^{2n}\bigl(k-m e^{i\frac{2\pi j}{2n+1}}\bigr).
\end{equation}
Then we deduce from \eqref{eqn-odd-1} that if $k\neq m$, then
\begin{equation}\label{eqn-odd-3}
    a_{2n+1}\prod_{j=1}^{2n}\bigl(k-m e^{i\frac{2\pi j}{2n+1}}\bigr)+a_{2n}\prod_{j=1}^{2n-1}\bigl(k-m e^{i\frac{2\pi j}{2n}}\bigr)+\cdots+a_1=v.
\end{equation}
Recall that $a_{2n+1}\neq 0$, the leading terms in \eqref{eqn-odd-3} are
\begin{align}\label{eqn-odd-4}
&\prod_{j=1}^{2n}\bigl(k-m e^{i\frac{2\pi j}{2n+1}}\bigr)= \prod_{j=1}^{n}\bigl|k-m e^{i\frac{2\pi j}{2n+1}}\bigr|^2\nonumber\\
 &= \prod_{j=1}^{n}\bigl(k^2+m^2-2km\cos\frac{2\pi j}{2n+1}\bigr)\gtrsim (m^2+k^2)^n.
\end{align}
If { $m,k$ solve \eqref{eqn-odd-3}, then $|m|, |k|$}  are bounded by some constant $ c(v,n,a_{2n+1})>0$. Thus \eqref{eqn-odd-3} has only finitely many possible integer solutions $m,k$. In other words, $|\Pi(v)|$ is finite for all $v\in\R$. Then the graph $G(v_1,v_2)$ is finite for all $v_1,v_2\in\R,v_1\neq v_2$ and therefore $g(v_1,v_2)<\infty$. To show $\gamma'(v)=\infty$ for $v\in\R$, it is sufficient to prove
\begin{equation*}
    \lim\limits_{\substack{m,k\to\infty\\m,k\in\Z, m\neq k}}|\bigl( f(k)-f(m)-v(k-m)\bigr)|=\infty.
\end{equation*}
In fact, if $m\neq k$, then $|k-m|\ge 1$. Thus, by  \eqref{eqn-odd-3} we obtain
\begin{equation*}
    |f(k)-f(m)-v(k-m)|\ge ( k^2+ m^2)^n \to \infty
\end{equation*}
as $m,k\to\infty$.

{\bf Case (2):} $\mathrm{deg}\, p$ is even. Assume that
$$
p(x)=a_{2n}x^{2n}+\cdots+a_1x+a_0, \quad a_{2n}\neq 0, n\geq 2.
$$
For every $v\in\R$,  the relation $\lambda_k(v)=\lambda_m(v),k,m\in\Z,k\neq m$ is equivalent to
\begin{equation}\label{eqn-even-1}
    a_{2n}\sum_{i+j=2n-1} k^{i}m^{j} + a_{2n-1}\sum_{i+j=2n-2} k^{i}m^{j}+\cdots+a_1=v.
\end{equation}
  It is related to the curve
\begin{equation}\label{eqn-even-2}
     a_{2n}\sum_{i+j=2n-1} x^{i}y^{j} + a_{2n-1}\sum_{i+j=2n-2} x^{i}y^{j}+\cdots+a_1=v
\end{equation}
in the plane $(x,y)\in \R^2$. We name it $C(v)$ and analyze its asymptotic behavior. Assume $C(v)$ has an asymptotic line, say $l:y=ax+b$ where the values of $a$ and $b$ are defined as
\begin{equation*}
    a=\lim\limits_{\substack{x \to \infty\\ (x,y)\in C(v)}} \frac{y}{x} \,\text{ and }\,    b:=\lim\limits_{\substack{x\to\infty\\ (x,y)\in C(v)}} y-ax.
\end{equation*}
It is easy to check that the asymptotic line $l$ of $C(v)$ can only be
\begin{equation}
    y=-x-\frac{a_{2n-1}}{a_{2n}}.
\end{equation}
Indeed, the line $l$ is the asymptotic line of $C(v)$ if \eqref{eqn-even-2} has no factor $y+x+\frac{a_{2n-1}}{a_{2n}}$ and $l\subset C(v)$ if \eqref{eqn-even-2} has the factor $y+x+\frac{a_{2n-1}}{a_{2n}}$. Notice that the asymptotic line $l$ is independent of the value of $v$.

 Since any two distinct $C(v_1)$ and $C(v_2)$ do not have common points, there exists at most one $v\in\R$ such that \eqref{eqn-even-2} has the factor $y+x+\frac{a_{2n-1}}{a_{2n}}$.

 We first prove $g(v_1,v_2)<\infty$. It is sufficient to prove that for any distinct $v_1,v_2\in\R$, the length of the alternative path in $G(v_1,v_2)$ cannot be arbitrary large. If not, given any $N\in\N$, we can choose an alternative path large enough such that there exists three adjacent vertices $k,k',k''\in\Z$ in the path and $|k|,|k'|,|k''|\ge N$. Without loss of generality, we assume $k$ and $k'$ are adjacent with the red-colored edge. Then we have $(k,k')\in C(v_1)$ and $(k'',k')\in C(v_2)$. Choose $N$ large enough such that the horizontal distance of points in $C(v_1)$ and $C(v_2)$ are smaller than $1$, then $k$ and $k''$ cannot be integers simultaneously, which leads to a contradiction.

Now we show that $\gamma'(v)=\infty$ for all $v\in\R$.
Fix $v\in\R$. By factorization we have
\begin{align*}
    &|\lambda_k(v)-\lambda_m(v)|=|p(k)-p(m)-v(k-m)|\\
    =&\bigl|(k-m)\Bigl( a_{2n}\sum\limits_{i=0}^{2n-1} k^im^{2n-1-i} +a_{2n-1}\sum\limits_{i=0}^{2n-2}k^im^{2n-1-i}+\cdots+a_1-v\Bigr)\bigr|.
\end{align*}
Thus $\lambda_k(v)\neq \lambda_m(v)$ is equivalent to $k\neq m$ and
\begin{equation}\label{equ-hkm}
    h(k,m):=a_{2n}\sum\limits_{i=0}^{2n-1} k^im^{2n-1-i} +a_{2n-1}\sum\limits_{i=0}^{2n-2}k^im^{2n-1-i}+\cdots+a_1-v\neq 0.
\end{equation}


To show $\gamma'(v)=\infty$, it is sufficient to prove
\begin{equation}\label{equ-gamma-infty-1}
    \lim\limits_{\substack{k,m\to \infty\\
    k\neq m\in\Z,
    h(k,m)\neq 0}}|(k-m)h(k,m)|=\infty.
\end{equation}
We consider the following two cases.


\noindent{\bf Case (a):} $\displaystyle |k+m|>M$, $M>0$ determined later. Factorizing it, we obtain
\begin{align*}
|h(k,m)|&\ge M|a_{2n}|\prod\limits_{j=1}^{n-1}\bigl(k^2-2\cos \frac{j\pi}{n}km+m^2\bigr)\\
     &\quad -|a_{2n-1}|\Bigl| \sum_{j=0}^{2n-2}k^jm^{2n-2-j} \Bigr|+ \text{ lower terms}\\
     &\geq (c_1|a_{2n}|M-c_2\left|a_{2n-1}\right|)(k^2+m^2)^{n-1}\to \infty
\end{align*}
when $k,m\to \infty$, provided that $M>\frac{c_2|a_{2n-1}|}{c_1|a_{2n}|}$.

\noindent{\bf Case (b):} $\displaystyle |k+m|\le M$. Let $\displaystyle m=-k+r,r=0, \pm 1,\pm 2,\cdots,\pm \lfloor M\rfloor$, and substitute it into $h(k,m)$. Then $h(k,m)=h(k,r-k)$ is a polynomial of $k$. If the polynomial has a degree $\mathrm{deg}\,h(k,r-k)\ge 1$, then $h(k,r-k)$ satisfy
\begin{equation}
    |h(k,r-k)|\to \infty
\end{equation}
when $k\to \infty$. If it is a constant, then we can write it as $h(k,r-k)=c-v$. By \eqref{equ-hkm} we obtain  $c-v\neq 0$. Then
\begin{equation}
    |(k-m)h(k,m)|=|(2k-r)(c-v)|\to\infty
\end{equation}
when $k\to\infty$ and $m=-k+r$.

The above discussion shows  that \eqref{equ-gamma-infty-1}  holds.
\begin{remark}\label{rem-2order}
From the proof of even case, we see if $p$ is a second order polynomial, we also have $\gamma'(v)=\infty$ for all $v\in \R$.
\end{remark}

\section{Applications}
\label{sec:5}

\subsection{Schr\"{o}dinger equations} Let
$$
p(k)=k^2, \quad k\in \N.
$$
Then the dispersive equation \eqref{eqn-abstract} reduces to the Schr\"{o}dinger equation \eqref{first}. It is easy to see that
$$
\Pi(v)=\O, v\in \R \backslash \Z; \quad \Pi(v) = \Z, v\in \Z.
$$
For every two integers $v_1\neq v_2$, the graph $G(v_1,v_2)$ consists of two two-way infinite paths, see Figure~\ref{fig:6}.
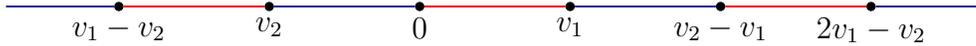
\begin{figure}[!htbp]
    \centering
    \vspace{0.04\textwidth}
\begin{tikzpicture}
       \SetVertexMath
    \GraphInit[vstyle=Simple]
    \tikzset{VertexStyle/.append style = {
minimum size = 3pt, inner sep =0pt}}
       \Vertex[x=0,y=0]{1}
       \Vertex[x=2,y=0]{2}
       \Vertex[x=4,y=0]{3}
       \Vertex[x=6,y=0]{4}
       \Vertex[x=8,y=0]{5}
       \Vertex[x=10,y=0]{6}
       \draw (1) node[below] {$v_1-v_2$};
       \draw (2) node[below] {$v_2$};
       \draw (3) node[below] {$0$};
       \draw (4) node[below] {$v_1$};
       \draw (5) node[below] {$v_2-v_1$};
       \draw (6) node[below] {$2v_1-v_2$};
       \draw[Red,thick,opacity=0.6] (1)--(2);
       \draw[Blue,thick,opacity=0.6] (2)--(3);
       \draw[Red,thick,opacity=0.6] (3)--(4);
       \draw[Blue,thick,opacity=0.6] (4)--(5);
       \draw[Red,thick,opacity=0.6] (5)--(6);
       \draw[Blue,thick,opacity=0.6] (-1.5,0)--(1);
       \draw[Blue,thick,opacity=0.6] (6)--(11.5,0);
\end{tikzpicture}
    \caption{An infinite path in $G(v_1,v_2)$ for the Schr\"{o}dinger equation.}
    \label{fig:6}
\end{figure}
Thus $g(v_1,v_2)=\infty$. Moreover, there is no two-colored circle in $G(v_1,v_2)$. By Remark \eqref{rem-2order}, we have $\gamma'(v)=\infty$. Then Proposition \ref{thm:main1} and Theorem~\ref{thm:main2} recovers Jaming and Komornik's work \cite{jaming2020moving} as follows.
\begin{proposition}[\cite{jaming2020moving}]
For any $T>0$, the following assertions hold for the Schr\"{o}dinger equation \eqref{first}.
\begin{itemize}
  \item [\rm{(1)}] Qualitative \eqref{eqn-qualitative} and quantitative \eqref{eqn-quantitative} observability holds on one line segment if and only if $v\in \R \backslash \Z$.
  \item [\rm{(2)}] If $v_1,v_2\in   \Z$ and $v_1\neq v_2$, the qualitative observability \eqref{eqn-qualitative} holds for $n=2$, but the quantitative observability \eqref{eqn-quantitative} fails.
\end{itemize}
\end{proposition}

In the sequel, we provide an example to show some differences between the second order Schr\"{o}dinger equation and higher order versions.  To this end, we consider a special higher-order linear Schr\"{o}dinger equation
\begin{equation}\label{eqn-higher-special}
    \partial_t u=i(-\partial_x^2u+\partial_x^4 u+\cdots+(-1)^{l}\partial_x^{2l}u), \quad u(0,x)=u_0(x)\in L^2(\T)
\end{equation}
where $l\in\N$ and $l>1$. It can be rewritten in the form \eqref{eqn-abstract} with
\begin{equation}
   p(k)=k^2+k^4+\cdots+k^{2l}.
\end{equation}

\begin{theorem}\label{thm-application-1}
Let $(t_1,x_1),(t_2,x_2)\in\R\times \T,v_1,v_2\in\R\times \T,v_1\neq v_2$  and $T>0$. Then the quantitative observability
\begin{equation}
    \int_\T | u_0 |^2\d x\lesssim \int_0^T\bigl(| u(t_1+t,x_1-v_1t) |^2+| u(t_2+t,x_2-v_2t) |^2\bigr)\d t
\end{equation}
holds for all solutions of \eqref{eqn-higher-special}.
\end{theorem}

\begin{proof}
Let $v\in\R$. The relation $\lambda_k(v)=\lambda_m(v)$ is equivalent to
\begin{equation*}
    \sum_{i=1}^{2l}k^{2l-i}m^{i-1}+\sum_{i=1}^{2l-2}k^{2l-2-i}m^{i-1}+\cdots +(k+m)=v
\end{equation*}
when $k\neq m$. This is related to the curve $C(v)$
\begin{equation}\label{eqn-higher-special-curve}
    \sum_{i=1}^{2l}x^{2l-i}y^{i-1}+\sum_{i=1}^{2l-2}x^{2l-2-i}y^{i-1}+\cdots+(x+y)=v.
\end{equation}
We claim that $y$ is a strictly decreasing function of $x$ along the curve $C(v)$. To see this,
we can rewrite \eqref{eqn-higher-special-curve} as
\begin{equation}\label{eqn-higher-1}
    \bigl(x+y\bigr)\biggl( \prod_{j=1}^{l-1}\bigl(x^2-2\cos \frac{j\pi}{l}xy+y^2\bigr)+\prod_{j=1}^{l-2}\bigl(x^2-2\cos\frac{j\pi}{l-1}xy+y^2\bigr)+\cdots+1 \biggr)=v.
\end{equation}

We first consider the case $v=0$, then the curve $C(v)$ reduces to $x+y=0$. Hence $y$ is strictly decreasing with $x$. If $v\neq 0$, then $x+y=0$ is the asymptotic line of $C(v)$. Assume $x>0$, write the equation of $C(v)$ as
\begin{equation*}
    x^{2l-1}\biggl( \Bigl(\frac{y}{x}\Bigr)^{2l-1}+ \Bigl(\frac{y}{x}\Bigr)^{2l-2}+\cdots+1  \biggr)+x^{2l-2}\biggl(\Bigl(\frac{y}{x}\Bigr)^{2l-3}+\cdots+1\biggr)+\cdots+1=v.
\end{equation*}
Then for any $(x_1,y_1),(x_2,y_2)\in C(v)$ and $x_2>x_1>0$, we obtain $y_2<y_1$ because of the strict increasing property of the function $x^{2m-1}+x^{2m-2}+\cdots+x+1,m\in\N$. This is also true for $x<0$.
Hence for any $v\in\R$, the curve $C(v)$ can be written as $y=g_v(x)$ which is a strictly decreasing function.

Now we choose any distinct $v_1,v_2\in\R$ and try to find an alternative cycle. Assume $v_2>v_1$ and any vertex $k_1\in G(v_1,v_2)$. If there exists a red-colored edge incident with $k_1$, then we must have an integer $k_2\neq k_1$ which satisfies $\lambda_{k_1}(v_1)=\lambda_{k_2}(v_1)$. This implies that $(k_1,k_2)\in C(v_1)$. If there exists a blue-colored edge incident with $k_2$, then we must have an integer $k_3\neq k_2$ and $k_3\neq k_1$ which satisfies $\lambda_{k_2}(v_2)=\lambda_{k_3}(v_2)$. This implies that $(k_3,k_2)\in C(v_2)$. If we can keep doing this $2j$ times so that $k_1,k_2,\cdots,k_{2j},k_{2j+1}=k_{1}$ can form an alternative cycle, then by the strict decreasing property of $g_v(x)$ we obtain
\begin{equation}
    k_1<k_3<\cdots<k_{2j+1},\label{eqn-higher-3}
\end{equation}
and
\begin{equation}
    k_2>k_4>\cdots>k_{2j}
\end{equation}
for all $i=1,2,\cdots,j$, see Figure~\ref{fig:7}.
\begin{figure}[!htbp]
    \centering
    \begin{tikzpicture}[scale=0.6]
        \draw[color=Red,thick,opacity=0.6] plot[domain=3:9.5,smooth] ({\x},{-\x +27/(\x*\x)});
        \draw[color=Blue,thick,opacity=0.6] plot[domain=4:9.5,smooth] ({\x},{-\x +64/(\x*\x)});
        \draw[very thick,->] (-0.03,0) -- (10,0) node[right] {$x$};
        \draw[very thick,->] (0,0) -- (0,-10);
        \draw (-0.2,-10) node[left] {$y$};
        \draw[color=Gray,dashed] plot[domain=0:9.5, smooth] (\x, -\x);
        \draw[color=Gray,dashed] (6,0)--(6,-5.25);
        \draw[color=Gray,dashed] (0,-5.25)--(6.7,-5.25);
        \draw[color=Gray,dashed] (6.7,-6.1)--(6.7,0);
        \draw[color=Gray,dashed] (7.3,-6.1)--(0,-6.1);
        \draw[color=Gray,dashed] (7.3,-6.8)--(7.3,0);
        \draw[color=Gray,dashed] (7.3,-6.8)--(0,-6.8);
        \draw (6,-5.25) node[circle, fill, inner sep=1pt] {};
        \draw (6.7,-5.25) node[circle,fill,inner sep=1pt] {};
        \draw (6.7,-6.1) node[circle,fill, inner sep =1pt] {};
        \draw (7.3,-6.1) node[circle,fill, inner sep =1pt] {};
        \draw (7.3,-6.8) node[circle,fill, inner sep =1pt] {};
        \draw (6,0) node[above] {\small $k_1$};
        \draw (6.7,0) node[above] {\small $k_3$};
        \draw (7.3+0.2,0) node[above] {\small $k_5$};
        \draw (0,-5.25) node[left] {\small $k_2$};
        \draw (0,-6.1) node[left] {\small $k_4$};
        \draw (0,-6.8) node[left] {\small $k_6$};
        \draw (0.15,-0.15) node[above left] { $o$};
        \draw (2.5,-3.2) node[rotate=-45] {\small $x+y=0$};
        \draw (7,-7.8) node {\small $C(v_1)$};
        \draw (9,-6.6) node {\small $C(v_2)$};
    \end{tikzpicture}
    \caption{Attempting to construct an alternative cycle.}
    \label{fig:7}
\end{figure}
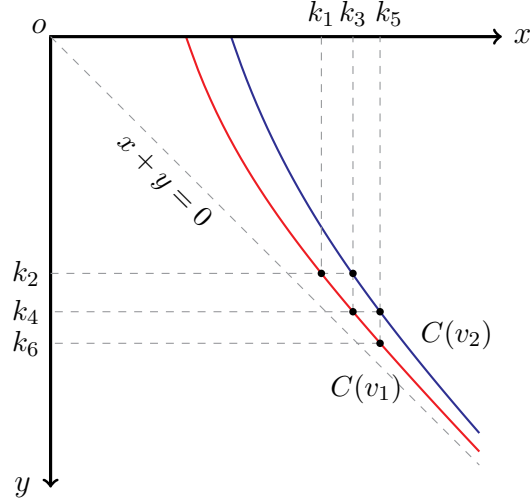
However \eqref{eqn-higher-3} contradicts to $k_{2j+1}=k_{1}$. Hence there exists no alternative cycle in $G(v_1,v_2)$. By Theorem~\ref{thm:main3} we conclude that the observability inequality holds.
\end{proof}

\subsection{The linear KdV equation}
The Korteweg-de Vries (KdV) equation reads
$$
    \partial_t u+u\partial_x u+\partial_x^3 u=0.
$$
It is introduced in \cite{Kort1895} to explain the observation of solitary waves in a shallow channel of water. Consider the linear KdV equation
\begin{equation}\label{eqn-linear-kdv}
    \partial_t u+\partial_x^3u=0, \quad u(0,x)=u_0(x)\in L^2(\T).
\end{equation}
It is a special case of \eqref{eqn-abstract} with $p(k)=k^3.$
 Hence by Theorem~\ref{thm:main3} the qualitative observability of \eqref{eqn-linear-kdv} with two moving points holds for any $T>0$ if and only if the quantitative observability holds.

Define a set
\begin{equation}
    \Gamma:=\bigl\lbrace k^2+km+m^2:k,m\in\Z \text{ and } k\neq m \bigr\rbrace.
\end{equation}

\begin{proposition}\label{thm-kdv-1}
Let $(t_0,x_0)\in \R\times \T$ and $v\in\R$.
\begin{enumerate}
    \item [\rm{(1)}] If $v\notin \Gamma$, then for every $T>0$, the quantitative observability
$$
    \int_\T | u_0 |^2\d x\lesssim \int_0^T | u(t_0+t,x_0-vt)|^2 \d t
$$
holds for all solutions of \eqref{eqn-linear-kdv}.
    \item [\rm{(2)}] If $v\in \Gamma$, then there exist non-trivial solutions satisfying
$$
        u(t_0+t,x_0-vt)=0,\quad \forall t\in\R.
$$
\end{enumerate}
\end{proposition}
\begin{proof}
This is a direct corollary of Proposition~\ref{thm-sec2-main-1}.
\end{proof}

By the definition of $\Gamma$, if $v\in\R \backslash \Z$, then $v\notin \Gamma$, and by {Proposition~\ref{thm-kdv-1} (1)} , the observability holds from one segment. The following theorem provide a criterion for the observability inequality from two segments when $v_1,v_1$ are integers.

\begin{theorem}\label{thm-application-2}
Let $(t_1,x_1),(t_2,x_2)\in \R\times \T,v_1,v_2\in \Z, v_1\neq v_2$. Then
$$
    \int_\T | u_0 |^2\d x\lesssim \int_0^T\bigl(| u(t_1+t,x_1-v_1t) |^2+| u(t_2+t,x_2-v_2t) |^2\bigr)\d t
$$
holds for all solutions of \eqref{eqn-linear-kdv} if $v_1$ and $v_2$ satisfies
\begin{enumerate}
    \item [\rm{(1)}] $v_2>4v_1>0$ or
    \item [\rm{(2)}] there exists a prime $p\equiv 2\pmod{3}$ such that $\mathrm{ord}_p(v_1)\neq \mathrm{ord}_p(v_2)$, where $\mathrm{ord}_p(v)$ denotes the largest integer $m$ such that $p^m \mid v$.
\end{enumerate}
\end{theorem}
\begin{proof}
(i) Assume that $v_2>4v_1>0$. By Theorem~\ref{thm:main3}, the quantitative observability holds if we can show the qualitative one, namely
$$
u(t_i+t,x_i-v_it)=0,t\in (0,T), i=1,2 \Longrightarrow   u_0\equiv 0.
$$
It suffices to show {that} every Fourier coefficient $c_k$ of $u_0$ is zero. In fact, if $k\notin \Pi(v_i)$ for some $i=1,2$, then we must have $c_k=0$ by Proposition~\ref{thm-sec2-main-1} (iv). Now we assume $k\in \Pi(v_1)\cap\Pi(v_2)$. According to the graph of $y=x^3-vx$, if $k>2\bigl( \frac{v}{3} \bigr)^{1 /2}$, then there exists no $m\in \Z$ such that $k^3-vk=m^3-vm$. Hence we only need to consider $k\in \bigl[-2\bigl( \frac{v}{3} \bigr)^{1 /2},2\bigl( \frac{v}{3} \bigr)^{1 /2}\bigr]$. Since $k\in \Pi(v_2)$ there exists another $k'\neq k$ such that $k'^3-v_2k'=k^3-v_2k$. This implies that
\begin{equation}
    |k'|\ge \Bigl( \frac{v_2}{3} \Bigr)^{1 /2}>2\Bigl(\frac{v_1}{3}\Bigr)^{1 /2}.
\end{equation}
Hence $k'\notin \Pi(v_1)$, see Figure~\ref{fig:8}.    \begin{figure}[htbp]
        \centering
         \begin{tikzpicture}[xscale=1.1,yscale=0.8]
         \draw[Red,thick,opacity=0.6] plot[domain=-1.8:1.8, smooth] (\x, \x*\x*\x-1*\x);
         \draw[Blue,thick,opacity=0.6] plot[domain=-2.1:2.1, smooth] (\x, \x*\x*\x-4.2*\x);
           \draw[very thick,->] (-3.5,0) -- (3.5,0) node[right] {$x$};
           \draw[very thick,->] (0,-4) -- (0,4) node[above] {$y$};
           \draw (-0.2,-0.01) node[below] {\small $0$};
         \draw[dashed] (0.7,0) node[above] {$k$} -- (0.7,-2.62);
         \draw[dashed] (0.7,-2.62) -- (1.6,-2.62);
         \draw[dashed] (1.6,-2.62) -- (1.6,2.53);
         \draw (1.5,0) node[above right] {$k'$};
         \draw (1.9,4) node[right] {$y=x^3-v_1x$};
         \draw (2.2,0.5) node[right] {$y=x^3-v_2x$};
        \end{tikzpicture}
        \caption{Graphs of $y=x^3-v_1x$ and $y=x^3-v_2 x$ with $v_2>4v_1>0$.}
        \label{fig:8}
    \end{figure}
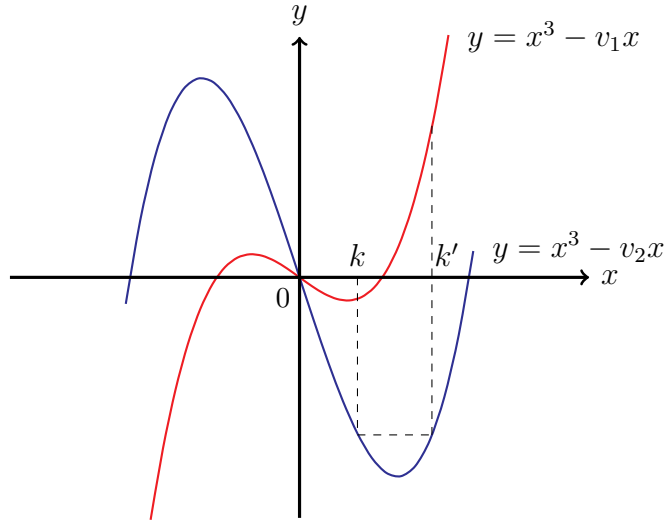
and then each $k\in \Pi(v_1)\cap\Pi(v_2)$ satisfy {the} conditions Lemma~\ref{thm-reduced-lemma}, thus $c_k=0$.

(ii) Without loss of generality, we assume that $\mathrm{ord}_p(v_1)>\mathrm{ord}_p(v_2)$, $v_1,v_2\in\Gamma$, and set $\mathrm{ord}_p(v_1)=2e,e\in\N$ (here we use the property that any integer $n$ can be written {in} the form $n=k^2+km+m^2,k,m\in\Z$ if and only if  $2\mid \mathrm{ord}_p(n)$ for all $p\equiv 2\pmod{3}$, see \cite[Proposition 9.14]{ireland1990classical}). We only need to consider the vertex $k\in \Pi(v_1)\cap\Pi(v_2)$. Since $k\in\Pi(v_1)$, there exists $m\neq k,m\in\Z$ such that
\begin{equation*}
    v_1=k^2+km+m^2.
\end{equation*}
Define $\omega:=e^{\frac{2\pi i}{3}}$, then we can decompose $v_1$ as
\begin{equation*}
    v_1=\bigl( k-m\omega \bigr)\bigl(k-m\overline{\omega}\bigr).
\end{equation*}
Notice that $p\equiv 2\pmod{3}$, i.e., $p$ is an Eisenstein prime which is irreducible in the integral domain $\Z [\omega]$. By the unique factorization theorem of $\Z[\omega]$, we obtain
\begin{equation*}
    p^e \mid k-m\omega \text{ and }p^e \mid k-m\overline{\omega},
\end{equation*}
which implies
\begin{equation}\label{eqn-linear-kdv-prime-1}
    p^e\mid k \text{ and }p^e \mid m.
\end{equation}
On the other hand, by $k\in \Pi(v_2)$ we know that there exists $l\neq k,l\in\Z$ such that
\begin{equation*}
    v_2=k^2+kl+l^2.
\end{equation*}
Then we have the decomposition
\begin{equation*}
    v_2=\bigl(k-l\omega\bigr)\bigl(k-l\overline{\omega}\bigr).
\end{equation*}
Since $\mathrm{ord}_p(v_1)>\mathrm{ord}_p(v_2)$, we have $p^{2e}\nmid v_2$, and thus
\begin{equation}\label{eqn-linear-kdv-prime-2}
    p^e\nmid k-l\omega \text{ and }p^e\nmid k-l\overline{\omega}.
\end{equation}
Then by $p^e\mid k$ in \eqref{eqn-linear-kdv-prime-1} and $p^e\nmid k-l\omega$ in \eqref{eqn-linear-kdv-prime-2} we obtain $p^e\nmid l$. Hence $l\notin \Pi(v_1)$. This implies that for each $k\in \Pi(v_1)\cap\Pi(v_2)$, there exists only red-colored edges. Therefore, the reduced graph $\Tilde{G}(v_1,v_2)$ is empty and we complete the proof.
\end{proof}

\section*{Acknowledgments}
We owe our sincere thanks to the referees for their insightful comments which have largely improved the presentation of this work.
This work is partially supported by the National Natural Science Foundation of China under grants No.12171442, 12171178.


\end{document}